\documentclass[a4paper,12pt]{amsart}

\usepackage{amsmath,amssymb,amsthm,graphicx,pstricks,url}
\usepackage[all]{xy}
\usepackage{lscape}
\usepackage{color}
\usepackage{hyperref}
\usepackage{tikz}
\usetikzlibrary{arrows}
 \usetikzlibrary{snakes}
 \usetikzlibrary{shapes}

\title{Cluster equivalence and graded derived equivalence}
\author{Claire Amiot}
\address{Institut de Recherche de Math\'ematique Avanc\'ee, 7 rue Ren\'e Descartes, 67084 Strasbourg Cedex, France}
\email{amiot@math.unistra.fr}
\author{Steffen Oppermann}
\address{Insitutt for matematiske fag,
Norges Teknisk-Naturvitenskapelige Universitet,
N-7491 Trondheim, Norway}
\email{steffen.oppermann@math.ntnu.no}

\thanks{Work on this project has been done while both authors were working at NTNU, Trondheim, supported by NFR Storforsk grant no.\ 167130.}

\newcommand{\Hom}{{\sf Hom }}
\newcommand{\RHom}{\mathbf{R}{\sf Hom }}
\newcommand{\End}{{\sf End }}
\newcommand{\Ext}{{\sf Ext }}

\newcommand{\mor}{{\sf mor\hspace{.02in} }}
\newcommand{\Coker}{{\sf Coker }}

\renewcommand{\mod}{{\sf mod \hspace{.02in}  }}
\newcommand{\Mod}{{\sf Mod \hspace{.02in} }}

\newcommand{\add}{{\sf add \hspace{.02in} }}
\newcommand{\ind}{{\sf ind \hspace{.02in} }}

\newcommand{\thick}{{\sf thick \hspace{.02in}  }}
\newcommand{\proj}{{\sf proj \hspace{.02in} }}

\newcommand{\gr}{{\sf gr \hspace{.02in} }}

\newcommand{\ten}{\otimes}
\newcommand{\lten}{\overset{\mathbf{L}}{\ten}}

\renewcommand{\lim}{{\sf lim \hspace{.02in}}}
\newcommand{\colim}{{\sf colim \hspace{.02in}}}
\newcommand{\pretr}{{\sf pretr \hspace{.02in}}}
\newcommand{\rep}{{\sf rep \hspace{.02in}}}
\newcommand{\Jac}{{\sf Jac}}

\newcommand{\Aa}{\mathcal{A}}

\newcommand{\Cc}{\mathcal{C}}
\newcommand{\Dd}{\mathcal{D}}
\newcommand{\Ee}{\mathcal{E}}
\newcommand{\Ff}{\mathcal{F}}
\newcommand{\Hh}{\mathcal{H}}

\newcommand{\Tt}{\mathcal{T}}

\newcommand{\Uu}{\mathcal{U}}
\newcommand{\Vv}{\mathcal{V}}
\newcommand{\Xx}{\mathcal{X}}
\newcommand{\Yy}{\mathcal{Y}}
\newcommand{\Zz}{\mathcal{Z}}
\newcommand{\SSS}{\mathbb{S}_2}

\newcommand{\ZZ}{\mathbb{Z}}

\newcommand{\bsm}{\begin{smallmatrix}}
\newcommand{\esm}{\end{smallmatrix}}

\newtheorem{thma}{Theorem}[section]

\newtheorem{lema}[thma]{Lemma}
\newtheorem*{lem*}{Lemme}

\newtheorem{cora}[thma]{Corollary}
\newtheorem*{prop*}{Proposition}
\newtheorem{prop}[thma]{Proposition}
\theoremstyle{remark}

\newtheorem{rema}[thma]{Remark}

\newtheorem{exa}[thma]{Example}
\theoremstyle{definition}

\newtheorem{dfa}[thma]{Definition}

\DeclareMathOperator{\gldim}{gl.\!dim}
\newcommand{\cov}[2]{{\rm Cov}(#1 , #2)}
\newcommand{\leftsub}[2]{{\vphantom{#2}}_{#1}{#2}}

\begin{document}

\begin{abstract}
In this paper we introduce a new approach for organizing algebras of global dimension at most 2. We introduce the notion of cluster equivalence for these algebras, based on whether their generalized cluster categories are equivalent. We are particularly interested in the question how much information about an algebra is preserved in its generalized cluster category, or, in other words, how closely  two algebras are related if they have equivalent generalized cluster categories.

Our approach makes use of the cluster-tilting objects in the generalized cluster categories: We first observe that cluster-tilting objects in generalized cluster categories are in natural bijection with cluster-tilting subcategories of derived categories, and then prove a recognition theorem for the latter.

Using this recognition theorem we give a precise criterion when two cluster equivalent algebras  are derived equivalent. For a given algebra we further describe all the derived equivalent algebras which have the same canonical cluster tilting object in their generalized cluster category.

Finally we show that in general, if two algebras are cluster equivalent, then (under certain conditions) the algebras can be graded in such a way that the categories of graded modules are derived equivalent. To this end we introduce mutation of graded quivers with potential, and show that this notion reflects mutation in derived categories.
\end{abstract}

\maketitle

\tableofcontents

\section{Introduction}
Tilting theory is an essential tool in representation theory of finite dimensional algebras. It permits to link algebras called derived equivalent, i.e.\ whose bounded derived categories are equivalent (see \cite{Hap} -- also see \cite{tilting} for a broader overview).

The introduction of cluster algebras goes back to Fomin and Zelevinsky \cite{FZ1}. The categorical interpretation of their combinatorics has been a crucial turn  in tilting theory and has brought new perspectives into the field: \emph{cluster-tilting theory}. The first step was the introduction of the cluster category $\Cc_Q$ associated with a finite acyclic quiver $Q$ in \cite{Bua} (using \cite{Mar}). The cluster category $\Cc_Q$ is defined as the orbit category $\Dd^b(kQ)/\SSS$ of the bounded derived category of the path algebra $kQ$ under the action of the autoequivalence $\SSS:=\mathbb{S}[-2]$ where $\mathbb{S}$ is the Serre functor of $\Dd^b(kQ)$. This category is triangulated (cf.\ \cite{Kel}), $\Hom$-finite (the $\Hom$-spaces are finite dimensional) and Calabi-Yau of dimension 2 (2-CY for short), that is there exists a functorial isomorphism $\Hom_{\Cc_Q}(X,Y)\simeq D\Hom_{\Cc_Q}(Y,X[2])$ for any objects $X$ and $Y$, where $D$ is the duality over $k$. The category $\Cc_Q$ has certain special objects called 
 \emph{cluster-tilting objects}. From one cluster-tilting object it is possible to construct others using a procedure called \emph{mutation}, whose combinatorics is very close to the combinatorics developed by Fomin and Zelevinsky for cluster algebras. The images of the tilting $kQ$-modules under the natural projection $\pi_Q \colon \Dd^b(kQ)\rightarrow \Cc_Q$ are cluster-tilting objects, and tilted algebras (= endomorphism algebras of tilting modules over a path algebra $kQ$) can be seen as specific quotients of cluster-tilted algebras (= endomorphism algebras of cluster-tilting objects in $\Cc_Q$).

The notion of cluster categories has been generalized in \cite{Ami3} replacing the hereditary algebra $kQ$ by an algebra $\Lambda$ of global dimension at most two. In this case the orbit category $\Dd^b\Lambda/\SSS$ is in general neither triangulated nor $\Hom$-finite. We restrict to the case where it is $\Hom$-finite (the algebra $\Lambda$ is then said to be $\tau_2$-finite). The generalized cluster category is now defined to be the triangulated hull $(\Dd^b(\Lambda)/\SSS)_\Delta$ in the sense of \cite{Kel} of the orbit category $\Dd^b\Lambda/\SSS$. The natural functor
$$\xymatrix{\pi_\Lambda \colon \Dd^b(\Lambda)\ar[r] & \Dd^b\Lambda/\SSS\ar@{^(->}[r] & (\Dd^b\Lambda/\SSS)_\Delta =:\Cc_\Lambda}$$ is triangulated. 

The aim of this paper is to study to what extent the derived categories are reflected by the cluster categories. More precisely we ask: If two algebras of global dimension at most two have the same cluster category, are they then automatically also derived equivalent? Since the answer to this question is negative in general, we further investigate how the two derived categories are related in case they are not equivalent.

Our means towards this goal is to study cluster-tilting objects. It has been shown in \cite{Ami3} that for any tilting complex $T$ in $\Dd^b(\Lambda)$ such that the endomorphism algebra $\End_{\Dd^b(\Lambda)}(T)$ is of global dimension~$\leq 2$, the object $\pi_{\Lambda}(T)$ in $\Cc_{\Lambda}$ is cluster-tilting.  Therefore cluster-tilting objects in cluster categories can be seen as analogs of tilting complexes in bounded derived categories.

It is an important result in tilting theory that tilting objects determine (in some sense) the triangulated category (\cite{Rickard}, or, more generally, \cite{Kel9}). Unfortunately, in cluster-tilting theory an analog of this theorem only exists for cluster categories coming from acyclic quivers (see \cite{Kel5}).

The first main result of this paper is to provide a `cluster-tilting' analog of  the abovementioned theorem from tilting theory. One key observation is that the endomorphism algebra $\widetilde{\Lambda} := \End_{\Cc_\Lambda}(\pi(\Lambda))$ has a natural grading given by $\widetilde{\Lambda}=\bigoplus_{p\in\ZZ}\Hom_{\Dd^b\Lambda}(\Lambda,\SSS^{-p}\Lambda)$. This graded algebra admits as $\ZZ$-covering the locally finite subcategory
\[ \Uu_\Lambda := \pi_\Lambda^{-1}(\pi_\Lambda(\Lambda))=\add\{\SSS^p\Lambda \mid p\in\ZZ\}\subset \Dd^b\Lambda, \]
which is a cluster-tilting subcategory of the derived category $\Dd^b \Lambda$.

\begin{thma}[Recognition theorem -- Theorem~\ref{clustertilt}] \label{recognition}
Let $\Tt$ be an algebraic triangulated category with a Serre functor and with a cluster-tilting subcategory $\Vv$. Let $\Lambda$ be a $\tau_2$-finite algebra with global dimension~$\leq 2$. Assume that there is an equivalence of additive categories with $\SSS$-action  $\xymatrix{f \colon \Uu_{\Lambda} \ar[r]^-\sim & \Vv}$. Then there exists a triangle equivalence $F \colon \Dd^b\Lambda\rightarrow \Tt$.
\end{thma}

In contrast to Keller's and Reiten's theorem, this result is a recognition theorem for the derived category. That is, we use cluster-tilting theory for studying a classical problem in representation theory. 

Applying Theorem~\ref{recognition} helps us to study the notion of \emph{cluster-equivalent} algebras, i.e.\ algebras of global dimension~$\leq 2$ with the same cluster category, which is the main subject of this paper. 

Using it, we give the following criterion on when two cluster equivalent algebras $\Lambda_1$ and $\Lambda_2$ with $\pi_{\Lambda_1} \Lambda_1\simeq\pi_{\Lambda_2} \Lambda_2$ in the common cluster category are derived equivalent.

\begin{thma}[Theorem~\ref{derivedeq1}] \label{derivedeq_intro}
In the above setup the cluster equivalent algebras $\Lambda_1$ and $\Lambda_2$ are derived equivalent if and only if the $\ZZ$-graded algebras $\widetilde{\Lambda}_1$ and $\widetilde{\Lambda}_2$ are graded equivalent. 
\end{thma}

We give a characterization of the tilting complexes $T$ of $\Dd^b(\Lambda_1)$ such that $\pi_{\Lambda_1}(T)\simeq \pi_{\Lambda_1}(\Lambda_1)$ (Theorem~\ref{characterizationinverseimage}), with which we can show that in case the equivalent conditions of Theorem~\ref{derivedeq_intro} hold, the algebras $\Lambda_1$ and $\Lambda_2$ are iterated $2$-APR tilts of one another.

In case the equivalent conditions of Theorem~\ref{derivedeq_intro} do not hold we then consider the case where the two different $\ZZ$-gradings on $\widetilde{\Lambda}_1\simeq\widetilde{\Lambda}_2$ are compatible (that is when they induce a $\ZZ^2$-grading on $\widetilde{\Lambda}_1$; an assumption that seems to be always satisfied in actual examples). In this case, the natural $\ZZ$-grading on $\widetilde{\Lambda}_2$ induces a $\ZZ$-grading on $\Lambda_1$ and vice versa. We then obtain the following result.

\begin{thma}[Theorem~\ref{maintheorem}] \label{gradedderivedeq}
Let $\Lambda_1$ and $\Lambda_2$ be cluster equivalent, and assume we are in the setup described above. Then there is a triangle equivalence
\[ \Dd^b(\gr\Lambda_1) \simeq \Dd^b(\gr\Lambda_2) \]
(where $\gr \Lambda_i$ denotes the category of graded $\Lambda_i$-modules).
\end{thma}

These results can be extended in the case where the cluster-tilting objects $\pi_{\Lambda_1}(\Lambda_1)$ and $\pi_{\Lambda_2}(\Lambda_2)$ are not isomorphic but are linked by a sequence of mutations. In order to extend them to this setup, we introduce a notion of left (and right) mutation. Indeed the mutation in the generalized cluster category can be lifted to a mutation in the derived category. Furthermore, by a result of Keller \cite{Kel10},  the cluster category $\Cc_{\Lambda}$ is equivalent to a cluster category associated to a quiver with potential. Hence, by a fundamental result of \cite{KY}, endomorphism algebras of cluster-tilting objects related to the canonical one by mutations are Jacobian algebras, and are related to one another by mutations of quivers with potential introduced in \cite{DWZ}. These two observations lead us to introduce the notion of left (and right) mutation of graded quiver with potential in order to give a combinatorial description of the left (and right) mutation in the derived category. 

 This graded mutation allows us to deduce a combinatorial way to prove that two algebras of global dimension at most 2 are derived equivalent.
  In particular, as a direct consequence of Theorem~\ref{derivedeq_intro}, we obtain the following result which is a generalization of a result due to Happel \cite{Hap87} stating that two path algebras are derived equivalent if and only if they are iterated reflections from one another.
  
  \begin{thma}[Corollary \ref{graded mutation and derived equivalence}]\label{graded mutation and derived equivalence intro}
Let $\Lambda_1$ and $\Lambda_2$ be two algebras of global dimension 2, which are $\tau_2$-finite. Assume that one can pass from the graded quiver with potential  associated with $\widetilde{\Lambda}_1$ to the graded quiver with potential associated with $\widetilde{\Lambda}_2$ using a finite sequence of left and right mutations.  Then the algebras $\Lambda_1$ and $\Lambda_2$ are derived equivalent.
\end{thma}
 
 In \cite{AO3} we obtain the converse of Theorem~\ref{graded mutation and derived equivalence intro} in the case where these algebras are cluster equivalent to hereditary algebras. There we also apply the results presented here to understand and describe algebras which are cluster equivalent to tame hereditary algebras.
 
\medskip

The paper is organized as follows:

Section~\ref{section_backgr} is devoted to recalling background on generalized cluster categories, cluster-tilting subcategories, and graded algebras.

In Section~\ref{section.recognition} we prove Theorem~\ref{recognition}.

This result is applied to Iyama-Yoshino reduction of derived categories in Section~\ref{section.IY}.

In Section~\ref{section_der-eq} we prove Theorem~\ref{derivedeq_intro}, giving a criterion which determines which algebras are derived equivalent among cluster equivalent ones. We further classify derived equivalent algebras having the same canonical cluster-tilting object.

We introduce mutation of graded quiver with potential in Section~\ref{section_left_mutation}. Using a result of Keller and Yang, we show that this notion gives a combinatorial description of mutation in derived categories.

In Section~\ref{section_triang_orbit} we recall and apply to our setup some results on triangulated orbit categories due to Keller.

Theorem~\ref{gradedderivedeq}, which exhibits gradings on cluster equivalent algebras making them graded derived equivalent, is shown in Section~\ref{section_gr_der_eq}.

\subsection*{Notation}
Throughout $k$ is an algebraically closed field and all algebras are finite dimensional $k$-algebras. For a finite-dimensional $k$-algebra $A$, we denote by $\mod A$ the category of finite-dimensional right $A$-modules. For an additive $k$-linear category $\Aa$ we denote by $\mod \Aa$ the category of finitely presented functors $\Aa^{\rm op}\rightarrow \mod k$.  
 By triangulated category we mean $k$-linear triangulated category satisfying the Krull-Schmidt property. For all triangulated categories we denote the shift functor by $[1]$.

\section{Background} \label{section_backgr}
This section is devoted to recalling results that will be used in this paper. We first give the definition of generalized cluster categories, and then state some results on cluster-tilting subcategories, and graded algebras.

\subsection{Generalized cluster categories}

Let $\Lambda$ be an algebra of global dimension~$\leq 2$. We denote by $\Dd^b(\Lambda)$ the bounded derived category
of finitely generated $\Lambda$-modules. It has a Serre functor that
we denote by $\mathbb{S}$. We denote by $\SSS$ the composition
$\mathbb{S}[-2]$, and by $\tau_2$ the composition $H^0\SSS$.

The \emph{generalized cluster category} $\Cc_\Lambda$ of $\Lambda$ has been defined in \cite{Ami3} as
the triangulated hull of the orbit category $\Dd^b(\Lambda)/\SSS$.  We
denote by $\pi_\Lambda$ the triangle functor

$$\xymatrix{\pi_\Lambda \colon \Dd^b(\Lambda)\ar@{->>}[r] &
  \Dd^b(\Lambda)/\SSS\ar@{^(->}[r] & \Cc_\Lambda}.$$
More details on triangulated hulls are given in Section~\ref{section_triang_orbit} (Example~\ref{defclustercat}).

\begin{dfa}
An algebra $\Lambda$ of global dimension~$\leq 2$ is said to be \emph{$\tau_2$-finite} if $\tau_2$ is nilpotent.
\end{dfa}

We set $\widetilde{\Lambda}:=\End_{\Cc_{\Lambda}}(\pi \Lambda)\simeq\bigoplus_{p\geq 0}\Hom_{\Dd^b(\Lambda)}(\Lambda,\SSS^{-p}\Lambda).$ The algebra $\Lambda$ is $\tau_2$-finite if and only if the algebra $\widetilde{\Lambda}$ is finite dimensional. In this case we have the following result:
\begin{thma}[{\cite[Theorem~4.10]{Ami3}}]
Let $\Lambda$ be an algebra of global dimension
$\leq 2$ which is $\tau_2$-finite. Then the generalized cluster category $\Cc_\Lambda$ is a $\Hom$-finite,
$2$-Calabi-Yau category.
\end{thma}

\subsection{Cluster-tilting subcategories}
\begin{dfa}[Iyama]
Let $\Tt$ be a triangulated category, which is $\Hom$-finite. A functorially finite subcategory $\Vv$ of $\Tt$ is \emph{cluster-tilting} (or $2$-cluster-tilting) if 
$$\Vv=\{X\in \Tt \mid \Hom_\Tt(X,\Vv[1])=0\}=\{X\in \Tt \mid \Hom_\Tt(\Vv,X[1])=0\}.$$
We will call an object $T$ of $\Tt$ \emph{cluster-tilting} if the category $\add(T)$ is cluster-tilting. If $\Tt$ is 2-Calabi-Yau, and $T$ is a cluster-tilting object, the endomorphism algebra $\End_\Tt(T)$ is called \emph{2-Calabi-Yau-tilted}. If the category $\Tt$ has a Serre functor $\mathbb{S}$, then we have  $\SSS\Vv=\Vv$ for any cluster-tilting subcategory $\Vv$ where $\SSS:=\mathbb{S}[-2]$.
\end{dfa}   

\begin{exa}\label{listexamples}
The following examples of cluster-tilting objects will be used in the rest of this paper:
\begin{enumerate}
\item Let $Q$ be an acyclic quiver. If $T\in \mod kQ$ is a tilting module, then $\pi_Q(T)\in \Cc_{Q}$ is a cluster-tilting object in the cluster category $\Cc_Q$ (\cite{Bua}).
\item Let $\Lambda$ be a $\tau_2$-finite algebra of global dimension~$\leq 2$.  Let $T\in\Dd^b(\Lambda)$ be a tilting complex such that $\End_\Dd(T)$ has global dimension~$\leq 2$, then $\Uu_T=\add\{\SSS^p T \mid p\in \mathbb{Z}\}$ is cluster-tilting in $\Dd^b\Lambda$ (\cite[Theorem~1.22]{Iya} or \cite[Proposition~5.4.2]{Ami2}). 
\item Let $\Lambda$ be a $\tau_2$-finite algebra of global dimension~$\leq 2$. Let $T\in\Dd^b(\Lambda)$ be a tilting complex such that $\End_\Dd(T)$ has global dimension~$\leq 2$, then $\pi_\Lambda(T)\in \Cc_\Lambda$ is a cluster-tilting object in $\Cc_\Lambda$ (\cite[Theorem~4.10]{Ami3}). 
\end{enumerate}
\end{exa}

\begin{prop}[\cite{Kel5}] \label{approximation}
Let $\Tt$ be a triangulated category with Serre functor $\mathbb{S}$ and $\Vv\subset\Tt$ be a cluster-tilting subcategory. Then for any $X$ in $\Tt$ there exists a triangle called \emph{approximation triangle}
$$\xymatrix{V_1\ar[r] & V_0\ar[r]^v & X\ar[r] & V_1[1]}$$ where $V_0$ and $V_1$ are objects in $\Vv$ and where $\xymatrix{v \colon V_0\ar[r] & X}$ is a minimal right $\Vv$-approxima-tion.
\end{prop}

The following result explains how cluster-tilting subcategories can be mutated.

\begin{thma}[{\cite[Theorem 5.3]{IY}}] \label{IyamaYoshino}
Let $\Tt$ be a triangulated category with Serre functor $\mathbb{S}$ and $\Vv\subset\Tt$ be a cluster-tilting subcategory. Let $X\in \Vv$ be indecomposable, and set
\[ \Vv' := \add ( \ind(\Vv) \setminus \{\SSS^p X \mid p\in\ZZ\}), \]
where $\ind(\Vv)$ denotes the indecomposable objects in $\Vv$. Then there exists a unique cluster-tilting subcategory $\Vv^*$ with $\Vv' \subseteq \Vv^* \neq \Vv$. Moreover
\[ \Vv^* = \add (\Vv' \cup \{\SSS^p X^L \mid p\in\ZZ\}) = \add (\Vv' \cup \{\SSS^p X^R \mid p\in \ZZ\}), \]
where $X^L$ and $X^R$ are obtained via triangles
$$\xymatrix{X\ar[r]^f &B\ar[r] & X^L\ar[r] & X[1]}\textrm{ and } \xymatrix{ X^R\ar[r] & B'\ar[r]^g & X\ar[r] & X^R[1]}$$
where $f$ (resp.\ $g$) is a minimal left (resp.\ right) $\Vv'$-approximation.
These triangles are called \emph{left and right exchange triangles}.
\end{thma}

\subsection{Basic results on graded algebras}

Let $G$ be an abelian group. (In this paper, $G$ will always be $\ZZ$ or $\ZZ^2$.) Let $\Lambda:=\bigoplus_{p\in G}\Lambda^p$ be a $G$-graded algebra. We denote by $\gr \Lambda$ the category of finitely generated graded modules over $\Lambda$. For a graded module $M=\bigoplus_{p\in G}M^p$, we denote by $M\langle q \rangle$ the graded module $\bigoplus_{p\in \ZZ}M^{p+q}$ (that is, the degree $p$ part of $M\langle q \rangle$ is $M^{p+q}$). The locally bounded subcategory
\[ \cov{\Lambda}{G} := \add\{\Lambda \langle p \rangle \mid p\in G\} \subseteq \gr \Lambda \]
is called \emph{the $G$-covering} of $\Lambda$.

\begin{thma}[\cite{GM}]
Let $\Lambda$ be a $G$-graded algebra. Then there is an equivalence $$\xymatrix{\mod \cov{\Lambda}{G} \ar[r]^(.65)\sim & \gr \Lambda}.$$
\end{thma}

\noindent
Here is a consequence of \cite[Theorem~5.3]{GG1}:

\begin{thma}[Gordon-Green]\label{gordongreen}
Let $\Lambda$ be an algebra with two different $G$-gradings. We denote by $\cov{\Lambda_1}{G}$ the $G$-covering corresponding to the first grading, and $\cov{\Lambda_2}{G}$ the $G$-covering corresponding to the second $G$-grading. Then the following are equivalent:
\begin{enumerate}
\item There is an equivalence $U \colon \mod \cov{\Lambda_1}{G} \overset{\sim}{\longrightarrow} \mod \cov{\Lambda_2}{G}$ such that the following diagram commutes.
$$\xymatrix{\mod \cov{\Lambda_1}{G} \ar[rr]^U\ar[dr] && \mod \cov{\Lambda_2}{G} \ar[dl] \\ & \mod
  \Lambda &}$$
\item There exist  $r_i\in G$ and an isomorphism of $G$-graded algebras
\[ \xymatrix{{\Lambda_2}\ar[r]^-{\sim}_-{G} & \bigoplus_{p\in G}\Hom_{\cov{\Lambda_1}{G}}(\bigoplus_{i=1}^nP_i\langle r_i\rangle,\bigoplus_{i=1}^nP_i\langle r_i+p\rangle) } \]
where $\Lambda_1 = \bigoplus_{i=1}^nP_i$ in $\gr \Lambda_1$.
\end{enumerate}
In this case we say that the gradings are \emph{equivalent}.
\end{thma}

\section{Cluster-tilting subcategories determine the derived category} \label{section.recognition}

\subsection{Bijection between cluster-tilting subcategories}

In this subsection we show that, for a $\tau_2$-finite algebra of global dimension~$\leq 2$, the projection functor induces a bijection between cluster-tilting subcategories of its cluster category and cluster-tilting subcategories of its derived category.

\begin{prop}\label{preimage}
Let $\Lambda$ be a $\tau_2$-finite algebra of global dimension~$\leq 2$. Let $T\in\Cc_\Lambda$ be a cluster-tilting object. Then $\pi^{-1}(T)\subset\Dd^b\Lambda$ is a cluster-tilting subcategory of $\Dd^b(\Lambda)$.
\end{prop}
\begin{proof}
Since the functor $\pi_\Lambda \colon \Dd^b(\Lambda)\rightarrow \Cc_\Lambda$ is triangulated, we get the inclusions
\begin{align*}
\pi^{-1}(T) & \subset \{X\in \Dd^b(\Lambda) \mid \Hom_{\Dd^b(\Lambda)}(X,\pi^{-1}(T)[1])=0\} \textrm{, and} \\ \pi^{-1}(T) & \subset  \{X\in \Dd^b(\Lambda) \mid \Hom_{\Dd^b(\Lambda)}(\pi^{-1}(T),X[1])=0\}
\end{align*}
Let $T\simeq T_1\oplus\cdots\oplus T_n$ be the decomposition on $T$ in indecomposable objects. For all $i=1,\ldots,n$, the object $T_i$ is rigid. Hence, by \cite{AO2}, there exists $T'_i\in \Dd^b(\Lambda)$ such that $\pi(T'_i)=T_i$. Now let $X$ be in $\{X\in \Dd^b(\Lambda)\mid \Hom_{\Dd^b(\Lambda)}(\pi^{-1}(T),X[1])=0\}$. Therefore for all $i=1,\ldots,n$ and all $p\in \ZZ$ the space $\Hom_{\Dd^b\Lambda}(\SSS^pT'_i,X[1])$ vanishes. Then for all $i=1,\ldots, n$ we have $$\Hom_{\Cc}(T_i,\pi(X)[1])=\Hom_{\Cc}(\pi(T'_i),\pi (X)[1]) \simeq \bigoplus_{p\in\ZZ}\Hom_{\Dd}(\SSS^pT'_i,X[1])=0.$$
Since $T$ is cluster-tilting, we have $\pi(X)\in \add(T)$, thus $X\in \pi^{-1}(T)$. The last inclusion is shown similarly.
\end{proof}

The following proposition shows that the converse is also true.

\begin{prop} \label{prop_image_ct_again}
 Let $\Lambda$ be a $\tau_2$-finite algebra of global dimension~$\leq 2$.
Let $\Vv$ be a cluster-tilting subcategory of $\Dd^b(\Lambda)$. Then $\pi_\Lambda(\Vv)$ is a cluster-tilting subcategory of $\Cc_{\Lambda}$.
\end{prop}

In the proof we will need the following piece of notation:

\begin{dfa}
For $\Uu$ and $\Vv$ subcategories of a triangulated category $\Tt$ we denote by $\Uu * \Vv$ the full subcategory of $\Tt$ consisting of objects $M$ of $\Tt$ appearing in a triangle $\xymatrix{U \ar[r]&  M  \ar[r]& V \ar[r] & U[1]}$  with $U \in \Uu$ and $V \in \Vv$.
\end{dfa}

One easily sees that $\Uu$ is a cluster-tilting subcategory of $\Tt$ if and only if $\Tt=\Uu*\Uu[1]$ and $\Hom_\Tt(\Uu,\Uu[1])=0$.

\begin{proof}[Proof of Proposition~\ref{prop_image_ct_again}]
Let $X$ and $Y$ be objects in $\Vv$. Then we have $$\Hom_{\Cc_\Lambda}(\pi(X),\pi(Y)[1])=\Hom_{\Cc_\Lambda}(\pi(X),\pi(Y[1]))=\bigoplus_{p\in\mathbb{Z}}\Hom_{\Dd^b(\Lambda)}(X,\SSS^pY[1])=0$$ since $\SSS^pY\in\SSS^p\Vv=\Vv$. Therefore we have $\Hom_{\Cc_\Lambda}(\pi(\Vv),\pi(\Vv)[1])=0$.

Denote by $\Uu:=\Uu_{\Lambda}=\pi^{-1}(\pi\Lambda)$ the canonical cluster-tilting subcategory of $\Dd^b(\Lambda)$. Then since $\Vv$ is cluster-tilting in $\Dd^b(\Lambda)$ we have $\Uu\subset \Vv[-1]*\Vv$. Since $\pi$ is a triangle functor, then we have the inclusions $$\pi(\Uu)\subset \pi(\Vv[-1]*\Vv)\subset \pi(\Vv)[-1]*\pi(\Vv).$$ Now since $\pi(\Uu)$ is a cluster-tilting subcategory of $\Cc_\Lambda$ we have:
$$\Cc_\Lambda=\pi(\Uu)*\pi(\Uu)[1]=(\pi(\Vv)[-1]*\pi(\Vv))*(\pi(\Vv)*\pi(\Vv)[1]).$$
Since $\Hom_{\Cc_\Lambda}(\pi(\Vv),\pi(\Vv)[1])=0$, we have $\pi(\Vv)*\pi(\Vv)=\pi(\Vv)$. Therefore, by associativity of $*$, we get $$\Cc_\Lambda=\pi(\Vv)[-1]*(\pi(\Vv)*\pi(\Vv))*\pi(\Vv)[1]=\pi(\Vv)[-1]*\pi(\Vv)*\pi(\Vv)[1].$$
Now let $X\in \Cc_\Lambda$ such that $\Hom_{\Cc_\Lambda}(X,\pi(\Vv)[1])=0$. There exists triangles
$$\xymatrix{ \pi(V_1)[-1]\ar[r] & \pi(V_0)[-1]\ar[r] & Y\ar[r] & \pi(V_1) \quad\textrm{and}\quad \pi(V_2)\ar[r] & Y\ar[r] & X\ar[r] & \pi(V_2)[1]}$$ with $V_0, V_1, V_2$ in $\Vv$. Since $\Hom_{\Cc_\Lambda}(X,\pi(\Vv)[1])=0$ the second triangle splits and we have $X\oplus \pi(V_2)\simeq Y$. Then $\Hom_{\Cc_\Lambda}(\pi(\Vv)[-1],Y)=0$ and the first triangle splits. Hence we have $\pi(V_1)\simeq Y\oplus \pi(V_0)\simeq X\oplus \pi(V_2)\oplus \pi(V_0)$
and $X\in \pi(\Vv)$. Therefore  $\pi(\Vv)$ is a cluster-tilting subcategory of $\Cc_\Lambda$.
\end{proof}

\subsection{Recognition theorem}

The aim of this subsection is to prove Theorem~\ref{clustertilt}. Since its setup is that of algebraic triangulated categories we recall the definition.

\begin{dfa}
A triangulated category $\Tt$ is called \emph{algebraic} if there is a Frobenius exact category $\mathcal{E}$ such that $\mathcal{T} = \underline{\Ee}$.
\end{dfa}

\begin{thma}\label{clustertilt}
Let $\Tt$ be an algebraic triangulated category with a Serre functor and with a cluster-tilting subcategory $\Vv$. Let $\Lambda$ be a $\tau_2$-finite algebra with global dimension~$\leq 2$. Assume that there is an equivalence of additive categories with $\SSS$-action  $\xymatrix{f \colon \Uu_{\Lambda} \ar[r]^-\sim & \Vv}$ (with $\Uu_{\Lambda}$ as in Example~\ref{listexamples}(2)). Then there exists a triangle equivalence $\xymatrix{F \colon \Dd^b\Lambda\ar[r] & \Tt}$ such that the following diagram commutes
$$\xymatrix{\Dd^b(\Lambda)\ar[r]^F & \Tt\\ \Uu_{\Lambda} \ar[r]^f \ar@{^(->}[u] & \Vv.\ar@{^(->}[u]}$$
\end{thma}

Our strategy for the proof is as follows: We introduce the category of radical morphisms $\mor \Vv$ (see Definition~\ref{def_morcat}) for a cluster-tilting subcategory $\Vv$. In the setup of Theorem~\ref{clustertilt} it follows that $\mor \Vv$ and $\mor \Uu_{\Lambda}$ are equivalent. We would like to complete the proof by saying that $\mor \Vv$ is equivalent to $\Tt$ (and similarly for $\Uu_{\Lambda}$), but unfortunately we do not have such an equivalence but just a bijection on objects (see Lemma~\ref{cone}). We will see that this bijection is nice enough to make the image of $\Lambda$ a tilting object in $\Tt$. Then the theorem follows from the tilting theorem.

\begin{dfa} \label{def_morcat}
Let $\Uu$ be a locally finite $k$-category. We define the category $\mor \Uu$. Objects are radical maps $U_1\rightarrow U_0$ in $\Uu$ and morphisms are commutative squares.
\end{dfa}

Let $u \colon U_1\rightarrow U_0$ be an object in $\Uu$. Since $\Uu$ is locally finite, the kernel $M$ of the map $$\xymatrix{0\ar[r] & M\ar[r] & D\Uu(U_1,-)\ar[r]^{u^*} & D\Uu(U_0,-)}$$ is in $\mod \Uu$ the category of finitely presented functors $\Uu^{\rm op}\rightarrow \mod k$. Hence there exists a map $h \colon H_1\rightarrow H_0$ such that $$\xymatrix{\Uu(-,H_1)\ar[r]^{h_*}  & \Uu(-,H_0)\ar[r] & M\ar[r] & 0}$$ is the minimal projective resolution of $M$. This map is uniquely defined up to isomorphism.
Therefore we can define the map $H \colon \mor \Uu\rightarrow \mor \Uu$ as $Hu=h$. Similarly, we define $H^- \colon \mor \Uu\rightarrow \mor \Uu$.

\begin{lema}
If $u\in \mor \Uu$ is left minimal then we have $H^-Hu\simeq u$. If $u\in\mor\Uu$ is right minimal then we have $HH^-u\simeq u$.
\end{lema}

\begin{proof}
The morphism $u \colon U_1\rightarrow U_0$ is left minimal if and only if the injective resolution $$\xymatrix{0\ar[r] & M\ar[r] & D\Uu(U_1,-)\ar[r]^{u^*} & D\Uu(U_0,-)}$$ is minimal, hence we get the result.
\end{proof}

Let $\Vv\subset \Tt$ and $\Lambda$ be as in Theorem~\ref{clustertilt}. Since $\widetilde{\Lambda}$ is finite-dimensional, the category $\Uu= \add\{\SSS^p\Lambda \mid p\in \mathbb{Z}\}$ is locally finite, hence so is the category $\Vv$. The autoequivalence $\SSS$ of $\Vv$ induces an autoequivalence of $\mor \Vv$ that we denote also by $\SSS$. Each map $v \colon V_1\rightarrow V_0$ decomposes in the direct sum of a left minimal map and a map of the form $[0\rightarrow V_2]$. Hence we can define a map $\Sigma \colon \mor \Vv\rightarrow \mor \Vv$ by
\[ \Sigma v =  \left\{ \begin{array}{ll} H\SSS^{-}v & \textrm{if } v \textrm{ is left minimal} \\ \left[ V_0 \rightarrow 0 \right] & \textrm{if } v=[0\rightarrow V_0]. 
\end{array} \right. \]
This is clearly a bijection whose inverse is
\[ \Sigma^- v = \left\{ \begin{array}{ll} \SSS H^- v &\textrm{if } v \textrm{ is right minimal} \\ \left[ 0 \rightarrow V_1 \right] & \textrm{if } v=[V_1\rightarrow 0].\end{array} \right. \]

\begin{lema} \label{cone}
Let $\Vv\subset\Tt$ and $\Lambda$ be as in Theorem~\ref{clustertilt}. Then the map $$\textsf{Cone} \colon \xymatrix@R=.1cm{\mor \Vv\ar[r] &\Tt\\ v\ar@{|->}[r] & \textsf{Cone}(v)}$$ is a bijection on isomorphism classes of objects of the categories $\mor \Vv$ and $\Tt$.
Moreover we have
\begin{enumerate}
\item $\textsf{Cone}(\Sigma v)\simeq (\textsf{Cone}(v))[1]$;
\item $\textsf{Cone} (\SSS v)\simeq \SSS(\textsf{Cone}(v))$.
\end{enumerate}
\end{lema}

\begin{proof}
If two objects $u$ and $v$ of $\mor \Vv$ are isomorphic, then their cones are isomorphic. Hence this map is well-defined on isomorphism classes of objects. By Proposition~\ref{approximation}, for each $T\in \Tt$ there exists a triangle $$\xymatrix{V_1\ar[r]^v & V_0\ar[r]^w & T\ar[r] & V_1[1]}.$$ The map $v$ is in the radical if and only if $w$ is a minimal right $\Vv$-approximation. Since minimal right approximations exists and are unique up to isomorphism, the map $\textsf{Cone}$ is bijective. 

Let $v \colon 0\rightarrow V_0$ be in $\mor \Vv$. Then we have
\begin{align*}
\textsf{Cone}(\Sigma v) & = \textsf{Cone}([V_0\rightarrow 0])\\
& = V_0[1]\\
& = (\textsf{Cone}(v))[1].
\end{align*}

Let $v \colon V_1\rightarrow V_0$ be left minimal in $\mor \Vv$. Let $h \colon H_1\rightarrow H_1$ be $H\SSS^-(v)$. Then we have an exact sequence in $\mod \Vv$:

$$\xymatrix{\Hom_\Tt(\Vv,H_1)\ar[r]^{h_*} & \Hom_\Tt(\Vv,H_0)\ar[r] & D\Hom_\Tt(\SSS^-V_1,\Vv)\ar[r]^{(\SSS^-v)^*} & D\Hom_\Tt(\SSS^-V_0,\Vv)}$$

By definition of $\SSS$ this sequence is isomorphic to 
$$\xymatrix{\Hom_\Tt(\Vv,H_1)\ar[r]^{h_*} & \Hom_\Tt(\Vv,H_0)\ar[r] & \Hom_\Tt(\Vv,V_1[2])\ar[r]^{v[2]^*} & D\Hom_\Tt(\Vv, V_0[2])}$$

Since $\Vv$ is cluster-tilting, the space $ \Hom_{\Tt}(\Vv, H_1[1])$ vanishes and the cokernel of $h_*$ is isomorphic to $\Hom_\Tt(\Vv,\textsf{Cone}(h))$. Since $\Hom_{\Tt}(\Vv, V_0[1])$ vanishes, the kernel of the map $v[2]^*$ is $\Hom_\Tt(\Vv,\textsf{Cone}(v)[1] )$. Hence we get 
\begin{align*}
\textsf{Cone}(\Sigma v)& = \textsf{Cone}(H\SSS^-v)\\
& = \textsf{Cone}(h)\\
& = \textsf{Cone}(v)[1]
\end{align*}
and we have (1).

Assertion (2) is immediate.
\end{proof}

\begin{lema}\label{isomorphism}
In the setup of Theorem~\ref{clustertilt}, let $u$ be in $\mor \Uu$. Then for all $p\in \mathbb{Z} $ we have isomorphisms
$$\Hom_\Tt(f\Lambda,\textsf{Cone}(fu)[p])\simeq \Hom_{\Dd^b\Lambda}(\Lambda, \textsf{Cone}(u)[p]).$$
\end{lema}

\begin{proof}
Let $\xymatrix{U_1\ar[r]^{u} & U_0\ar[r] & X\ar[r] & U_0[1]}$ be a triangle in $\Dd^b\Lambda$ with $u\in \mor \Uu$.  
Let $$\xymatrix{U_1^p\ar[r]^{u^p} & U_0^p\ar[r] & X[p]\ar[r] & U_0^p[1]}$$ be a $\Uu$-approximation triangle of $X[p]$ in $\Dd^b(\Lambda)$ with $u^p\in\mor \Uu$.  Then we have
\begin{align*}
\textsf{Cone} (f(u^p))&\simeq \textsf{Cone}(f(\Sigma_\Uu^p(u))) && \textrm{by Lemma~\ref{cone}(1)}, \\
& \simeq \textsf{Cone} (\Sigma_\Vv^p(fu)) && \textrm{since } f \textrm{ is an equivalence of } \SSS\textrm{-categories}\\
&\simeq \textsf{Cone}(fu)[p] && \textrm{by Lemma~\ref{cone}(1)}.
\end{align*}

Thus we have a triangle in $\Tt$
$$\xymatrix{f(U_1^p)\ar[r]^{fu^p} &f(U_0^p)\ar[r] & \textsf{Cone}(fu)[p]\ar[r] & (fU_1^p)[1]}$$
which gives an exact sequence
$$\xymatrix{\Hom_\Tt(f\Lambda,fU_1^p)\ar[r] & \Hom_{\Tt}(f\Lambda,fU_0^p)\ar[r] & \Hom_\Tt(f\Lambda,\textsf{Cone}(fu)[p])\ar[r] & \Hom_\Tt(f\Lambda,fU_1^p[1])}.$$
The space $\Hom_\Tt(f\Lambda,fU_1^p[1])$ vanishes since $f\Uu=\Vv$ is a cluster-tilting subcategory of $\Tt$.
But since $f$ is an equivalence we have 
$$\xymatrix{\Hom_\Tt(f\Lambda,fU_1^p)\ar[r] & \Hom_{\Tt}(f\Lambda,fU_0^p)\ar[r] & \Hom_\Tt(f\Lambda,\textsf{Cone}(fu)[p])\ar[r] & 0\\
\Hom_{\Dd^b\Lambda}(\Lambda,U_1^p)\ar[u]^\wr\ar[r] & \Hom_{\Dd^b(\Lambda)}(\Lambda,U_1^p)\ar[r]\ar[u]^\wr & \Hom_{\Dd^b\Lambda}(\Lambda,\textsf{Cone}(u)[p])\ar[r] & \Hom_{\Dd^b\Lambda}(\Lambda,U_1^p[1])=0 }$$
Hence we get
\[ \Hom_\Tt(f\Lambda,\textsf{Cone}(fu)[p])\simeq \Hom_{\Dd^b\Lambda}(\Lambda, \textsf{Cone}(u)[p]). \qedhere \]
\end{proof}

\begin{proof}[Proof of Theorem~\ref{clustertilt}]

Applying Lemma~\ref{isomorphism} to $u=[0\rightarrow \Lambda]$ we get for each $p\in \mathbb{Z}$ $$\Hom_\Tt(f\Lambda,f\Lambda[p])\simeq \Hom_{\Dd^b\Lambda}(\Lambda, \Lambda[p])=0.$$
Therefore the object $f\Lambda$ is a tilting object in the category $\Tt$. 

We will use the following theorem  which can be deduced from \cite[Theorem~8.5]{tiltingKeller}:

\begin{thma}[Keller]
Let $\Tt$ be a $\Hom$-finite algebraic triangulated category. Let $T\in\Tt$ be
a tilting object of $\Tt$, i.e.\ for any $i\neq 0$ the space $\Ext^i_\Tt(T,T)$
vanishes. Denote by $\Lambda$ the endomorphism algebra $\End_\Tt(T)$
and assume it is of finite global dimension. Then there exists an
algebraic equivalence
$F \colon \Dd^b \Lambda\rightarrow\thick_\Tt(T) $ sending the object $\Lambda$ on $T$ where $\thick_\Tt (T)$ is the smallest triangulated subcategory of $\Tt$ containing $T$ and stable under direct summands.
\end{thma}
Hence we have an equivalence
$$\xymatrix{\Dd^b(\Lambda)\ar[r]^(.4)\sim & \thick(f\Lambda)\subset \Tt}$$
where $\thick(f\Lambda)$ is the thick subcategory of $\Tt$ generated by $f(\Lambda)$. It remains to show that $\thick(f \Lambda) = \Tt$. Since $\Lambda$ has finite global dimension it suffices to show that the only object $Y \in \Tt$ such that $\Hom_{\Tt}(f\Lambda, Y[p])=0 \, \forall p \in \ZZ$ is $0$. So let $Y\in\Tt$ be
 such that for each $p\in\mathbb{Z}$ the space $\Hom_\Tt(f\Lambda,Y[p])$ vanishes. Form an approximation triangle in $\Tt$
$$\xymatrix{V_1\ar[r]^v & V_0\ar[r] & Y\ar[r] & V_1[1]}.$$ Since $f \colon \Uu\rightarrow \Vv$ is an equivalence and since $v$ is in $\mor \Vv$, there exists $u \colon U_0\rightarrow U_1\in \mor \Uu$ such that $f(u)=v$. Denote by $X$ the cone of $u$. Then by Lemma~\ref{isomorphism} we have an isomorphism
$$ \Hom_{\Dd^b\Lambda}(\Lambda, X[p]) \simeq \Hom_\Tt(f\Lambda,Y[p])=0\textrm{ for all p }\in \mathbb{Z}.$$
Hence we have $X=0$ and $U_0=U_1=0$. Therefore we have $Y=0$. Hence we have an equivalence $\xymatrix{\Dd^b(\Lambda)\ar[r]^(.55)\sim & \Tt}.$

It remains to prove that the following diagram is commutative.
\[ \xymatrix{\Dd^b(\Lambda)\ar[r]^(.55)\sim_F & \Tt\\ \Uu\ar[r]_f^\sim \ar@{^(->}[u] & \Vv\ar@{^(->}[u]} \]
Since $F(\Lambda)=f(\Lambda)$ and since $\Uu$ is the cluster-tilting subcategory $\add\{_\Lambda\SSS^p\Lambda \mid p\in \mathbb{Z}\}$, it is enough to prove that the functor $F$ commutes with $\SSS$. This is clear by the uniqueness of the Serre functor in a triangulated category.
\end{proof}

\section{Application to Iyama-Yoshino reduction} \label{section.IY}

In this section, as an application of the recognition theorem (Theorem~\ref{clustertilt}), we show that certain Iyama-Yoshino reductions of derived categories are derived categories again.

For lightening the writing, in this section we denote by $\Tt(X,Y)$ the space of morphisms $\Hom_\Tt(X,Y)$ in the category $\Tt$. If $\Uu$ is a subcategory of $\Tt$, we denote by $[\Uu](X,Y)$ the space of morphism in $\Tt$ between $X$ and $Y$ factorizing through an object in $\Uu$. If $\Tt$ is a triangulated category with Serre functor $\leftsub{\Tt}{\mathbb{S}}$, we set $\leftsub{\Tt}{\SSS} = \leftsub{\Tt}{\mathbb{S}}[-2]$, and simply write $\SSS$ for $\leftsub{\Tt}{\SSS}$ when there is no danger of confusion.

\subsection{Iyama-Yoshino reduction}
This subsection is devoted to recalling some results of \cite{IY}.

Let $\Dd$ be a triangulated $k$-category which is $\Hom$-finite and
with a Serre functor $_\Dd\mathbb{S}$. Let $\Uu$ be a full rigid (i.e.\ $\Dd(\Uu, \Uu[1]) = 0$) functorially finite subcategory of $\Dd$ which is stable under $\leftsub{\Dd}{\SSS} = \leftsub{\Dd}{\mathbb{S}}[-2]$. We define the full subcategory $\Zz$ of $\Dd$ by
\[ \Zz=\{X\in\Dd\mid \Dd(\Uu,X[1])=0\}. \]
We denote by $\Tt$ the category $\Zz/[\Uu]$. Its objects are those of $\Zz$ and for $X$ and $Y$ in $\Zz$ we have
\[ \Tt(X,Y):=\Dd(X,Y)/[\Uu](X,Y). \]

For $X$ in $\Zz$, let $X\rightarrow U_X$ be a left
$\Uu$-approximation. We define $X\{ 1\}$ to be the cone
$$\xymatrix{X\ar[r] & U_X\ar[r] & X\{ 1 \}\ar[r] & X[1]}.$$

\begin{rema}
In \cite{IY}, Iyama and Yoshino write $X\!\left<1\right>$ instead of $X\{1\}$. Here we deviate from their notation because in this paper pointy brackets are used to denote degree shifts.
\end{rema}

\begin{thma}[\cite{IY}] \label{IyamaYoshinoreduction}
The category $\Tt$ is  triangulated, with shift functor
$\{ 1 \}$ and Serre functor $\leftsub{\Tt}{\mathbb{S}} = \leftsub{\Dd}{\SSS}\{ 2
\}$. Moreover there is a 1-1 correspondence between cluster-tilting subcategories of $\Dd$ containing
$\Uu$ and cluster-tilting subcategories in $\Tt$.
\end{thma}

In this construction, for any triangle $\xymatrix{X\ar[r] & Y\ar[r] &
  Z\ar[r] & X[1]}$ in $\Dd$ such that $X$, $Y$ and $Z$ are in $\Zz$, we have a morphism of triangles 
$$\xymatrix{X\ar[r]^u\ar@{=}[d] & Y\ar[r]^v\ar[d] & Z\ar[r]\ar[d]^{w} &
  X[1]\ar@{=}[d] \\X\ar[r] & U_X\ar[r] & X\{ 1 \}\ar[r] &
  X[1]}$$
Then the image of $\xymatrix{X\ar[r]^u & Y\ar[r]^v & Z\ar[r]^{w} & X\{
  1\} }$ in
$\Tt$ is a triangle.

Here we need the following version of Iyama-Yoshino reduction for the setup of algebraic triangulated categories.

\begin{prop} \label{prop.IYalg}
In the setup of Theorem~\ref{IyamaYoshinoreduction}, if $\Dd$ is algebraic triangulated then so is~$\Tt$.
\end{prop}

\begin{proof}
Since $\Dd$ is algebraic triangulated we have $\Dd = \underline{\Ee}$ for some Frobenius exact category $\Ee$. We denote by $\Ff$ the preimage of $\Zz$ in $\Ee$. Since $\Ff$ is closed under extensions $\Ff$ is an exact category (whose exact sequences are those exact sequences in $\Ee$ which lie entirely in $\Ff$).

Let $\Vv$ be the preimage of $\Uu$ in $\Ee$. Then clearly $\Vv \subseteq \Ff$. We claim that the objects in $\Vv$ are projective. Indeed if we have a short exact sequence
\[ \xymatrix{ F_1 \ar@{>->}[r] & F_2 \ar@{->>}[r] & V } \]
in $\Ff$ with $V \in \Vv$ then, since $\Hom_{\Dd}(V, F_1[1]) = 0$ by definition of $\Zz$, the sequence splits. Similarly the objects in $\Vv$ are also injective.

Since $\Vv$ contains all projective-injective objects in $\Ee$ one sees that for any $F \in \Ff$ the right $\Vv$-approximation $\xymatrix@-.3cm{V \ar[r] & F}$ is an admissible epimorphism in $\Ee$. One easily checks that its kernel is again in $\Ff$, so that the approximation is also an admissible epimorphism in $\Ff$. Hence $\Ff$ has enough projectives, and these are precisely the objects in $\Vv$. Dually $\Ff$ has enough injectives, which are again the objects in $\Vv$.

Thus $\Ff$ is Frobenius exact, and
\[ \Tt = \Zz / [\Uu] = \Ff / [\Vv] = \underline{\Ff} \]
is algebraic triangulated.
\end{proof}

\begin{rema}
The proof of Proposition~\ref{prop.IYalg} is also a simpler proof for Theorem~\ref{IyamaYoshinoreduction} in the case when $\Dd$ is algebraic triangulated.
\end{rema}

\subsection{Reduction of the derived category}
Let $\Lambda=kQ/I$ be a $\tau_2$-finite algebra of global dimension~$\leq 2$.
Let $i_0\in Q_0$ be a source of $Q$ and $e:=e_{i_0}$ be the associated primitive idempotent of $\Lambda$. We apply Iyama-Yoshino's construction for
$\Uu_e=\add\{\SSS^{p}(e\Lambda) \mid p\in \ZZ\}\subset\Dd^b(\Lambda)$. 
That is, we denote by $\Zz$ the full subcategory of $\Dd:=\Dd^b(\Lambda)$ 
$$\Zz:=\{X\in\Dd\mid  \Dd(\SSS^p(e
\Lambda),X[1])=0\quad \forall p\in \ZZ\}.$$
Then, by Proposition~\ref{prop.IYalg}, the category 
$\Tt=\Zz/[\Uu_e]$
is algebraic triangulated.

Denote by $\Lambda'$ the algebra $\Lambda/\Lambda e\Lambda\simeq
(1-e) \Lambda (1-e)$. Since $i_0$ is a source of the quiver $Q$, the algebra $\Lambda$ is a one point extension of $\Lambda'$, namely
\[ \Lambda= \begin{bmatrix} \Lambda' & (1-e)\Lambda e \\ 0 & k \end{bmatrix}. \]
Then the projective $\Lambda$-modules are 
\[ (1-e)\Lambda=\begin{bmatrix} \Lambda' & (1-e)\Lambda e\end{bmatrix} \quad \textrm{and} \quad e\Lambda= \begin{bmatrix} 0 & k\end{bmatrix},\] and the injective $\Lambda$-modules are
$$(1-e)D\Lambda=\begin{bmatrix} D\Lambda' & 0 \end{bmatrix} \quad \textrm{and} \quad e D\Lambda= \begin{bmatrix} eD\Lambda (1-e) & k\end{bmatrix}.$$

\begin{lema}
The algebra $\Lambda'$ is a $k$-algebra of global dimension~$\leq 2$.
\end{lema}

\begin{proof}
Since $i_0$ is a source of $Q$, for $i\in Q_0$ with $i\neq i_0$ the minimal injective resolution of the simple $S_i$ in $\mod \Lambda$ does not contain the injective module $eD\Lambda$. Therefore, using the description of injectives above, this injective resolution can be seen as an injective resolution in $\mod \Lambda'$.
\end{proof}

The aim of this section is to prove the following theorem:

\begin{thma}\label{reductiontheorem}
There is a triangle equivalence  $\Dd':=\Dd^b(\Lambda ')\simeq\Tt:=\Zz/[\Uu_e]$.
\end{thma}

We first prove several lemmas.

\begin{lema} \label{p=-1}
 We have an isomorphism $\Dd'(\SSS\Lambda',\Lambda')\simeq \Dd( \SSS((1-e)\Lambda), (1-e)\Lambda)/[e\Lambda].$
\end{lema}

\begin{proof}
We choose a projective resolution of $D\Lambda'$ in $\mod \Lambda'$. 
\[ \tag{$*$} \xymatrix{0\ar[r] & P_2\ar[r] & P_1\ar[r] & P_0\ar[r] & D\Lambda'\ar[r] & 0}. \]
Since $\Lambda$ is a one point extension of $\Lambda'$ there is a short exact sequence
\[ \tag{$\dagger$} \xymatrix{0\ar[r] & (1-e)\Lambda e\ar[r]& (1-e)\Lambda\ar[r] & \Lambda'\ar[r] & 0} \]
in $\mod (\Lambda'^{\rm op}\ten \Lambda)$, where $(1-e)\Lambda e$ is the $\Lambda'$-$\Lambda$-bimodule $\begin{bmatrix} 0 & (1-e)\Lambda e\end{bmatrix}$. Note that as $\Lambda$-module this is just $e\Lambda^{\dim_k (1-e)\Lambda e}$. Applying $P_i\ten_{\Lambda'}-$ to $(\dagger)$ for $i=0,1,2$ we obtain short exact sequences
\[ \xymatrix@R=.4cm{ 0 \ar[r] & P_i \ten_{\Lambda'} (1-e) \Lambda e \ar@{=}[d] \ar[r] & P_i \ten_{\Lambda'} (1-e) \Lambda \ar@{=}[d] \ar[r] & P_i \ar[r] & 0 \\ & P_i  \Lambda e & P_i  \Lambda }. \]
Inserting these in $(*)$ we obtain the following projective resolution of the $\Lambda$-module $D\Lambda'$:
\[ \xymatrix@-.1cm{0\ar[r] &  P_2  \Lambda e \ar[r] & { \begin{matrix} P_2  \Lambda \\ \oplus \\ P_1  \Lambda e \end{matrix} } \ar[r] & { \begin{matrix} P_1  \Lambda \\ \oplus \\ P_0  \Lambda e \end{matrix} } \ar[r] & P_0  \Lambda \ar[r] & D\Lambda' \ar[r] & 0} \]
Since $\Lambda$ is of global dimension~$\leq 2$, the map $\xymatrix@-.3cm{P_2  \Lambda e \ar[r] & P_1  \Lambda e}$ is a split monomorphism, hence we can write 
\[ \xymatrix{0 \ar[r] & { \begin{matrix} P_2  \Lambda \\ \oplus \\ e \Lambda^m \end{matrix} } \ar[r] & { \begin{matrix} P_1  \Lambda \\ \oplus \\ P_0  \Lambda e \end{matrix} } \ar[r] & P_0  \Lambda \ar[r] & D\Lambda'\ar[r] &0} \]
for some $m \in \mathbb{N}$. Since $e$ is attached to a source of the quiver $Q$, the space $[\Uu_e]((1-e)\Lambda,(1-e)\Lambda)$ vanishes and we have $$\Dd(\SSS((1-e)\Lambda), (1-e)\Lambda)/[e\Lambda]=\Coker (\Dd(P_1  \Lambda, (1-e)\Lambda)\rightarrow \Dd(P_2  \Lambda, (1-e)\Lambda)).$$
Since $i_0$ is a source of the quiver of $\Lambda$, if $i,j\neq i_0$ we have $\Dd(e_i\Lambda,e_j\Lambda)=\Dd'(e_i\Lambda',e_j\Lambda').$ Hence we have 
\begin{align*}
\Dd(\SSS((1-e)\Lambda), (1-e)\Lambda)/[e\Lambda] & \simeq \Coker (\Dd(P_1 \Lambda, (1-e)\Lambda)\rightarrow \Dd(P_2 \Lambda, (1-e)\Lambda))\\
& \simeq \Coker (\Dd(P_1, \Lambda')\rightarrow \Dd(P_2, \Lambda')) \\
& \simeq \Dd'(\SSS\Lambda',\Lambda'). \qedhere
\end{align*}
\end{proof}

\begin{lema}\label{composition}
For any $p \geq 1$ the composition map
\[ \xymatrix{ \Tt(\SSS^{-p+1}\Lambda,\SSS^{-p}\Lambda) \ten_{\Lambda} \cdots \ten_{\Lambda} \Tt(\SSS^{-1}\Lambda,\SSS^{-2}\Lambda) \ten_{\Lambda} \Tt(\Lambda,\SSS^{-1}\Lambda) \ar[r] & \Tt(\Lambda,\SSS^{-p}\Lambda) } \]
is an isomorphism.
\end{lema}

\begin{proof}
By definition of $\Tt$ for any $X, Y \in \Tt$ we have an exact sequence\
\[ \xymatrix{ [\Uu_e](X,Y) \ar[r] & \Dd(X,Y) \ar[r] & \Tt(X,Y) \ar[r] & 0.} \]
Hence we obtain the following diagram

\[ \xymatrix{0&0\\ \Tt(\SSS^{-p+1}\Lambda,\SSS^{-p}\Lambda)\ten_{\Lambda}\ldots\ten_\Lambda\Tt(\Lambda,\SSS^{-1}\Lambda)\ar[r]\ar[u] & \Tt(\Lambda,\SSS^{-p}\Lambda)\ar[u]\\ \Dd(\SSS^{-p+1}\Lambda,\SSS^{-p}\Lambda)\ten_\Lambda\ldots \ten_{\Lambda} \Dd(\Lambda,\SSS^{-1}\Lambda)\ar[u]\ar[r] & \Dd(\Lambda,\SSS^{-p}\Lambda) \ar[u] \\ {*} \ar[r]\ar[u] & [\Uu_e](\Lambda,\SSS^{-p}\Lambda)\ar[u]}\]
with exact columns, and with
\[ * = \bigoplus_{j=1}^p\Dd(\SSS^{-p+1}\Lambda,\SSS^{-p}\Lambda)\ten_{\Lambda} \cdots \ten_\Lambda [\Uu_e](\SSS^{-j+1}\Lambda, \SSS^{-j}\Lambda)\ten_\Lambda\ldots \ten_{\Lambda} \Dd(\Lambda,\SSS^{-1}\Lambda) \]

The surjectivity of the composition map\[ \xymatrix{ \Tt(\SSS^{-p+1}\Lambda,\SSS^{-p}\Lambda) \ten_{\Lambda} \cdots \ten_{\Lambda} \Tt(\Lambda,\SSS^{-1}\Lambda) \ar[r] & \Tt(\Lambda,\SSS^{-p}\Lambda) } \]  is now  consequence of the following result:
\begin{lema}[\cite{Ami3}]
 The composition map
\[ \xymatrix{ \Dd(\SSS^{-p+1}\Lambda,\SSS^{-p}\Lambda) \ten_{\Lambda} \cdots \ten_{\Lambda} \Dd(\Lambda,\SSS^{-1}\Lambda) \ar[r] & \Dd(\Lambda,\SSS^{-p}\Lambda) } \]
 is an isomorphism.
\end{lema}
We now prove that the map
\begin{align*}
& \bigoplus_{j=1}^p\Dd(\SSS^{-p+1}\Lambda,\SSS^{-p}\Lambda)\ten_{\Lambda} \cdots \ten_\Lambda [\Uu_e](\SSS^{-j+1}\Lambda, \SSS^{-j}\Lambda)\ten_\Lambda\ldots\ten_{\Lambda}  \Dd(\Lambda,\SSS^{-1}\Lambda) \\
& \qquad \qquad \xymatrix{\ar[r] & [\Uu_e](\Lambda,\SSS^{-p}\Lambda)}
\end{align*}
is surjective.
Any morphism in $[\Uu_e](\Lambda, \SSS^{-p} \Lambda)$ is a sum of morphisms factoring through various $\SSS^{-q} e \Lambda$, with $0 \leq q \leq p$. Since the right radical $\add \{ \SSS^i \Lambda \mid i \in \ZZ\}$-approximation of $\SSS^{-q} e \Lambda$ lies in $\add \SSS^{-q+1} \Lambda$, and the left radical $\add \{\SSS^i \Lambda \mid i \in \ZZ\}$-approximation of $\SSS^{-q} e \Lambda$ lies in $\add \SSS^{-q} \Lambda$, we have that any map $\Lambda \rightarrow \SSS^{-p} \Lambda$ factoring through $\SSS^{-q} e \Lambda$ lies in the image of
\[\xymatrix{\Dd(\SSS^{-p+1}\Lambda,\SSS^{-p}\Lambda)\ten_{\Lambda} \cdots \ten_\Lambda [\Uu_e](\SSS^{-q+1}\Lambda, \SSS^{-q}\Lambda)\ten_\Lambda\ldots\ten_{\Lambda}  \Dd(\Lambda,\SSS^{-1}\Lambda)\ar[r] &[\Uu_e](\Lambda,\SSS^{-p}\Lambda)} \]
Therefore, using the above diagram, the composition map
\[ \xymatrix{ \Tt(\SSS^{-p+1}\Lambda,\SSS^{-p}\Lambda) \ten_{\Lambda} \cdots \ten_{\Lambda} \Tt(\Lambda,\SSS^{-1}\Lambda) \ar[r] & \Tt(\Lambda,\SSS^{-p}\Lambda) } \]
is an isomorphism.
\end{proof}

\begin{lema}\label{clustertiltcoincide}
For any $p\in \ZZ$,
we have $\Dd'(\Lambda',\SSS^p\Lambda')\simeq \Tt(\Lambda, \SSS^p\Lambda).$
\end{lema}
 \begin{proof}
For $p\geq 0$ both side vanishes since $\Lambda$ and $\Lambda'$ are of global dimension at most 2.
The case $p=-1$ is Lemma~\ref{p=-1} since we have $$[\Uu_e]((1-e)\Lambda,\SSS^{-1}(1-e)\Lambda)=[\SSS^{-1}e\Lambda]((1-e)\Lambda,\SSS^{-1} (1-e)\Lambda).$$

Using Lemmas~\ref{p=-1} and~\ref{composition} we show the assertion for any $p\leq -1$ by an easy induction and using the fact that $$\Tt(\SSS^{-1}\Lambda,\SSS^{-2}\Lambda)\ten_{\Lambda}\Tt(\Lambda,\SSS^{-1}\Lambda)\simeq \Tt(\SSS^{-1}(1-e)\Lambda,\SSS^{-2}(1-e)\Lambda)\ten_{\Lambda'}\Tt((1-e)\Lambda,\SSS^{-1}(1-e)\Lambda).$$
\end{proof}

\begin{rema}
This lemma can also be proved using Theorem~\ref{keller}, but we think it is good to also have a direct proof.
\end{rema}

\begin{proof}[Proof of Theorem~\ref{reductiontheorem}]
The strategy of the proof is to use the recognition theorem (Theorem~\ref{clustertilt}). The category $\Uu_\Lambda=\add\{ _\Lambda \SSS^p\Lambda \mid \ p\in \ZZ\}$ is a cluster-tilting subcategory of $\Dd$ which contains $\Uu_e$. Therefore, by Theorem~\ref{IyamaYoshinoreduction}, its image under the natural functor $\Zz\rightarrow \Zz/[\Uu_e]=\Tt$ is a cluster-tilting subcategory of $\Tt$. By Lemma~\ref{clustertiltcoincide} the category $\Uu_\Lambda/[\Uu_e]$ is equivalent to the category $\Uu_{\Lambda'}=\add\{ _{\Lambda'} \SSS^p\Lambda' \mid \ p\in \ZZ\} \subset \Dd'=\Dd^b(\Lambda')$ as category with $\SSS$-action.
Therefore, by Theorem~\ref{clustertilt}, we get an triangle equivalence $\Dd'\simeq \Tt$.
\end{proof} 

\begin{rema}
\begin{enumerate}
\item Theorem~\ref{reductiontheorem} also holds if $i_0$ is a sink of the quiver of $\Lambda$.
\item This result is related to \cite[Theorem~7.4]{Kel10}, where the author proves that the Iyama-Yoshino reduction of the generalized cluster category associated with a Jacobi-finite quiver with potential at a vertex is again a generalized cluster category associated with a Jacobi-finite quiver with potential.
\end{enumerate}
\end{rema}

\section{Cluster equivalent algebras: the derived equivalent case} \label{section_der-eq}

In this section we give a criterion for two cluster equivalent algebras to be derived equivalent. The main tool for proving this criterion is the recognition theorem (Theorem~\ref{clustertilt}). Further we study derived equivalent algebras satisfying the assumption that the canonical cluster-tilting objects in the common cluster category are isomorphic. We show that these algebras are all iterated $2$-APR tilts of one another.

\subsection{Derived equivalence is graded equivalence}

\begin{dfa}
Two $\tau_2$-finite algebras $\Lambda_1$ and $\Lambda_2$ of global dimension~$\leq 2$
 will be called \emph{cluster equivalent} if there
exists a triangle equivalence between their generalized cluster
categories $\Cc_{\Lambda_1}$ and $\Cc_{\Lambda_2}$.
\end{dfa}

\begin{prop}\label{derivedeqclustereq2}
Let $\Lambda_1$ and $\Lambda_2$ be $\tau_2$-finite algebras of global dimension~$\leq 2$. If $\Lambda_1$ and $\Lambda_2$ are derived equivalent, then there exists an equivalence $F^{\Dd} \colon \Dd^b(\Lambda_1) \rightarrow \Dd^b(\Lambda_2)$ which induces an equivalence $F \colon \Cc_{\Lambda_1} \rightarrow \Cc_{\Lambda_2}$ as in the following diagram.
\[ \xymatrix{
\Dd^b(\Lambda_1)\ar[r]^{F^\Dd}\ar[d]^{\pi_1} & \Dd^b(\Lambda_2)\ar[d]^{\pi_2} \\
\Cc_{\Lambda_1}\ar[r]^F & \Cc_{\Lambda_2}
} \]
In particular derived equivalent algebras are cluster equivalent.
\end{prop}
\noindent
We refer to Section~\ref{section_triang_orbit} for a formal proof of this proposition (Corollary~\ref{derivedeqclustereq}).

\begin{rema}
As we will see in the examples later, algebras which are not derived equivalent can still be cluster equivalent. 
\end{rema}

As a  consequence of Theorem~\ref{clustertilt} we have a first version of a criterion for cluster equivalent algebras to be derived equivalent. 
\begin{cora}
Let $\Lambda_1$ and $\Lambda_2$ be two $\tau_2$-finite algebras of global dimension
$\leq 2$ which are cluster equivalent.  Denote by $\pi_1$
(resp.\ $\pi_2$) the canonical functor $\Dd^b(\Lambda_1)\rightarrow
\Cc_{\Lambda_1}$ (resp.\ $\Dd^b(\Lambda_2)\rightarrow
\Cc_{\Lambda_2}$). Then the following are equivalent
\begin{enumerate}
\item $\Lambda_1$ and $\Lambda_2$ are derived equivalent;

\item there exists an $\SSS$-equivalence between the categories
  $\pi_2^{-1}(F\pi_1 \Lambda_1)\subset \Dd^b\Lambda_2$ and $\pi_1^{-1}(\pi_1 \Lambda_1)\subset\Dd^b\Lambda_1$ for some triangle equivalence  $F \colon \Cc_{\Lambda_1}\rightarrow \Cc_{\Lambda_2}$. 
\end{enumerate}
\end{cora}

\begin{proof}
 $(1)\Rightarrow (2)$:  By Proposition~\ref{derivedeqclustereq2} there exists a triangle equivalence $F^\Dd$ which induces a triangle equivalence $\xymatrix@-.3cm{F \colon \Cc_{\Lambda_1}\ar[r] & \Cc_{\Lambda_2}}$ such that the diagram
\[ \xymatrix{\Dd^b(\Lambda_1)\ar[r]^{F^\Dd}\ar[d]^{\pi_1} & \Dd(\Lambda_2)\ar[d]^{\pi_2} \\ \Cc_{\Lambda_1}\ar[r]^F & \Cc_{\Lambda_2}} \]
commutes. Therefore we have  $\SSS$-equivalences
\begin{align*}
\pi_2^{-1}(F\pi_1\Lambda_1) & = \pi_2^{-1}(\pi_2 F^\Dd\Lambda_1) \\
& = \add\{\SSS^p F^\Dd\Lambda_1 \mid p\in \ZZ\}  \\
& \simeq \add\{\SSS^p\Lambda_1 \mid p\in\ZZ\}  && \text{(by uniqueness of the Serre functor)}
\\
& \simeq \pi_1^{-1}(\pi_1\Lambda_1).
\end{align*}

$(2)\Rightarrow (1)$ Since $F$ is a triangle equivalence, and since $\pi_1\Lambda_1$ is cluster-tilting in $\Cc_{\Lambda_1}$, then $F\pi_1\Lambda_1$ is cluster-tilting in $\Cc_{\Lambda_2}$. Then by Proposition~\ref{preimage} the subcategory $\pi_2^{-1}(F\pi_1\Lambda_1)$ is cluster-tilting in $\Dd^b\Lambda_2$. Hence, by Theorem~\ref{clustertilt}, we get the result.
\end{proof}

\begin{rema}
It is not clear that the $F^\Dd$ constructed in the proof of $(2)\Rightarrow (1)$ commutes with $F$, but it induces a (possibly different) triangle equivalence $\xymatrix{\Cc_{\Lambda_1}\ar[r] & \Cc_{\Lambda_2}}$.
\end{rema}

\begin{thma}\label{derivedeq1}
Let $\Lambda_1$ and $\Lambda_2$ be two $\tau_2$-finite algebras of global dimension~$\leq 2$. For $i=1,2$ we denote by $\pi_i$
 the canonical functor $\Dd_i\rightarrow
\Cc_i$, where $\Dd_i:=\Dd^b(\Lambda_i)$ and $\Cc_i:=\Cc_{\Lambda_i}$. Denote by $\cov{\widetilde{\Lambda}_i}{\ZZ}$ the $\ZZ$-covering of the $\ZZ$-graded algebra $$\widetilde{\Lambda}_i:=\bigoplus_{p\geq 0}\Hom_{\Dd_i}(\Lambda_i,\SSS^{-p}\Lambda_i).$$ Assume that we have an isomorphism of algebras $\xymatrix{\widetilde{\Lambda}_1\ar[r]^-\sim_-f & \widetilde{\Lambda}_2.}$ 
Then the following are equivalent
\begin{enumerate}
\item there exists a derived equivalence $\xymatrix{F^\Dd \colon \Dd_1\ar[r]^-\sim & \Dd_2}$ such that the induced triangle equivalence $F \colon \Cc_{1}\rightarrow \Cc_{2}$ satisfies $F(\pi_1\Lambda_1)=\pi_2\Lambda_2$;
\item there is a  equivalence $\xymatrix{f^\ZZ \colon \mod\cov{\widetilde{\Lambda}_1}{\ZZ}\ar[r]^-\sim & \mod\cov{\widetilde{\Lambda}_2}{\ZZ}}$ extending $f$. 
\end{enumerate}
In this case the algebras $\Lambda_1$ and $\Lambda_2$ are cluster equivalent, and we have a commutative diagram:
$$\xymatrix{\Dd_1\ar[d]_{\pi_1}\ar[r]^\sim_{F^\Dd} & \Dd_2\ar[d]^{\pi_2}\\ 
\Cc_{1}\ar[r]^\sim_F\ar[d]_{\Hom_{\Cc_1}(\pi_1(\Lambda_1),-)} & \Cc_{2}\ar[d]^{\Hom_{\Cc_2}(\pi_2(\Lambda_2),-)}\\
\mod \widetilde{\Lambda}_1\ar[r]^\sim_{f} &\mod \widetilde{\Lambda}_2\\
 \mod \cov{\widetilde{\Lambda}_1}{\ZZ}\ar[u] \ar[r]^\sim_{f^\ZZ} & \mod\cov{\widetilde{\Lambda}_2}{\ZZ}\ar[u] }$$
\end{thma}

\begin{proof}
$(1)\Rightarrow (2)$: Assume Condition~(1) is satisfied. Then there exists a tilting complex $T\in \Dd^b\Lambda_1$ such that $\End_{\Dd^b\Lambda_1}(T)\simeq \Lambda_2$ and that the following diagram commutes $$\xymatrix{\Dd^b\Lambda_1\ar[d]^{\pi_1}\ar[rrr]^{\RHom_{\Dd^b\Lambda_1}(T,-)} &&& \Dd^b\Lambda_2\ar[d]^{\pi_2}\\ 
\Cc_{\Lambda_1}\ar[rrr]^\sim_F &&& \Cc_{\Lambda_2}}$$
Since $F\pi_1(\Lambda_1)$ is isomorphic to $\pi_2(\Lambda_2)=\pi_2(\RHom(T,T))$, we have $\pi_1(\Lambda_1)=\pi_1(T)$. So $T$ can be written $\bigoplus_{i=1}^n\SSS^{-d_i} e_i\Lambda_1$ for certain $d_i \in \ZZ$, where $\Lambda_1=\bigoplus_{i=1}^n e_i\Lambda_1$ is the decomposition into indecomposable projective modules. Then we have the following isomorphisms of $\ZZ$-graded algebras 
\begin{align*}
\widetilde{\Lambda}_2  & \underset{\ZZ}{\simeq} \bigoplus_{p \in \ZZ}\Hom_{\Dd_2}(\Lambda_2, \SSS^{-p}\Lambda_2)\\
&\underset{\ZZ}{\simeq} \bigoplus_{p \in \ZZ} \Hom_{\Dd_1}(T,\SSS^{-p}T)\\
&\underset{\ZZ}{\simeq} \bigoplus_{p \in \ZZ}\Hom_{\Dd_1}(\bigoplus_{i=1}^n\SSS^{-d_i} e_i\Lambda_1, \bigoplus_{i=1}^n\SSS^{-d_i-p} e_i\Lambda_1)\\
& \underset{\ZZ}{\simeq} \bigoplus_{p \in \ZZ} \Hom_{\cov{\widetilde{\Lambda}_1}{\ZZ}}(\bigoplus_{i=1}^n e_i\widetilde{\Lambda}_1\langle d_i\rangle, \bigoplus_{i=1}^n e_i\widetilde{\Lambda}_1\langle d_i + p\rangle),
\end{align*}
where $e_i\widetilde{\Lambda}_1$ is the projective $\widetilde{\Lambda}_1$-graded module $\bigoplus_{p\in \ZZ}\Hom_{\Dd_1}(\Lambda_1 ,\SSS^{-p} e_i\Lambda_1)$. 
Therefore, by Theorem~\ref{gordongreen}, we have a commutative diagram
$$\xymatrix{\mod\widetilde{\Lambda}_1\ar[r]^\sim & \mod\widetilde{\Lambda}_2\\ \mod\cov{\widetilde{\Lambda}_1}{\ZZ}\ar[r]^\sim \ar[u] & \mod \cov{\widetilde{\Lambda}_2}{\ZZ}\ar[u]}$$ and we get $(2)$.

$(2)\Rightarrow (1)$: By assumption we have a commutative diagram 
$$\xymatrix{\mod\widetilde{\Lambda}_1\ar[r]^\sim & \mod\widetilde{\Lambda}_2\\ \mod\cov{\widetilde{\Lambda}_1}{\ZZ}\ar[r]^\sim \ar[u] & \mod\cov{\widetilde{\Lambda}_2}{\ZZ}\ar[u]}$$ where the upper equivalence comes from the isomorphism of algebras $\widetilde{\Lambda}_1\simeq \widetilde{\Lambda}_2$.
Then, by Theorem~\ref{gordongreen}, there exists integers $d_i$ such that $\widetilde{\Lambda}_2$ is isomorphic as graded algebra to 
\[ \widetilde{\Lambda}_2 \underset{\ZZ}{\simeq} \bigoplus_{p\in\ZZ}\Hom_{\cov{\widetilde{\Lambda}_1}{\ZZ}}(\bigoplus_{i=1}^ne_i\widetilde{\Lambda}_1\langle d_i\rangle,\bigoplus_{i=1}^ne_i\widetilde{\Lambda}_1\langle d_i+p\rangle),\]
where $\widetilde{\Lambda}_1=\bigoplus_{i=1}^ne_i \widetilde{\Lambda}_1$ is the decomposition of $\widetilde{\Lambda}_1$ into indecomposables. Thus we have $e_i\widetilde{\Lambda}_1\simeq \bigoplus_{p\in \ZZ}\Hom_{\Dd_1}(\Lambda_1,\SSS^{-p}e_i \Lambda_1)$. Therefore we have $$\widetilde{\Lambda}_2\underset{\ZZ}{\simeq}\bigoplus_{p\in\mathbb{Z}}\Hom_{\Dd_1}(\bigoplus_{i=1}^n\SSS^{-d_i}e_i \Lambda_1,\bigoplus_{i=1}^n\SSS^{-d_i-p}e_i \Lambda_1).$$   This isomorphism of graded algebras means that we have an equivalence of cluster-tilting subcategories with $\SSS$-action
$$\xymatrix{ \Uu_1=\add\{\SSS^{-d_i-p} e_i \Lambda_1, i=1,\ldots,n,\ p\in\mathbb{Z}\} \ar[r]_(.6)u^(.6)\sim & \add\{\SSS^{-p}\Lambda_2 \mid p\in\mathbb{Z}\}=\Uu_2}$$
sending $\bigoplus_{i=1}^n\SSS^{-d_i} e_i \Lambda_1$ to $\Lambda_2$. 

Therefore, by Theorem~\ref{clustertilt}, we get a triangle equivalence $F^\Dd$ between $\Dd^b\Lambda_1$ and $\Dd^b\Lambda_2$ making the following square commutative.
$$\xymatrix{\Uu_1\ar[r]^u \ar@{^(->}[d] & \Uu_2\ar@{^(->}[d]\\ \Dd^b\Lambda_1\ar[r]^{F^\Dd}& \Dd^b\Lambda_2}$$
By Proposition~\ref{derivedeqclustereq2}, the functor $F^\Dd$ induces a triangle equivalence $F \colon \Cc_{\Lambda_1}\rightarrow\Cc_{\Lambda_2}$ such that the following diagram commutes,
$$\xymatrix{ \Dd^b\Lambda_1\ar[r]^{F^\Dd}\ar[d]^{\pi_1}& \Dd^b\Lambda_2\ar[d]^{\pi_2}\\ \Cc_{\Lambda_1}\ar[r]^F & \Cc_{\Lambda_2}}$$
and we have 
\[ F\pi_1\Lambda_1 \simeq F\pi_1(\bigoplus_{i=1}^n\SSS^{-d_i} e_i \Lambda_1) \simeq \pi_2 F^\Dd(\bigoplus_{i=1}^n\SSS^{-d_i}e_i \Lambda_1) \simeq \pi_2\Lambda_2. \]
This completes the proof of the implication $(2) \Rightarrow (1)$.

Moreover the square 
$$\xymatrix{\Cc_{\Lambda_1}\ar[r]^F\ar[d]_{\Hom_{\Cc_1}(\pi_1\Lambda_1, -)} & \Cc_{\Lambda_2}\ar[d]^{\Hom_{\Cc_2}(\pi_2\Lambda_2, -)} \\ \mod \widetilde{\Lambda}_1\ar[r]^f & \mod\widetilde{\Lambda}_2}$$
is commutative, and we get the commutative diagram of the theorem. 
\end{proof}

\begin{exa}\label{example}
Let $\Lambda_1=kQ^1/I^1$, $\Lambda_2=kQ^2/I^2$ and $\Lambda_3=kQ^3/I^3$ be the algebras given by the following quivers:
$$Q^1=\xymatrix@-.5cm{ &2&\\ 1\ar[ur]^a&&3\ar@<.5mm>[ll]^d\ar@<-.5mm>[ll]_c}, \quad Q^2=\xymatrix@-.5cm{ &2\ar[dr]^b&\\ 1&&3\ar@<.5mm>[ll]^d\ar@<-.5mm>[ll]_c} \textrm{ and } Q^3=\xymatrix@-.5cm{ &2\ar[dr]^b&\\ 1\ar[ur]^a&&3\ar[ll]^d}$$
with relations $I^1=\langle ac\rangle$, $I^2=\langle cb\rangle$ and $I^3=\langle ba\rangle$.
It is easy to check that the algebras $\widetilde{\Lambda}_i$, for $i=1,2,3$ are all isomorphic to the Jacobian algebra $\Jac(\widetilde{Q},W)$ (see Section~\ref{subsection jacobian} for definition) where
$$\widetilde{Q}= \xymatrix@-.5cm{ &2\ar[dr]^b&\\ 1\ar[ur]^a&&3\ar@<.5mm>[ll]^d\ar@<-.5mm>[ll]_c} \textrm{ and } W=cba.$$
The algebras $\Lambda_1$ and $\Lambda_2$ are derived equivalent since they are both derived equivalent to the path algebra of a quiver of type $\tilde{A}_2$.
The quiver $Q^3$ contains an oriented cycle, therefore the algebra $\Lambda_3$ is not derived equivalent to a hereditary algebra. 
We compute the $\ZZ$-coverings with respect to the different gradings:
$$\xymatrix@-.5cm{&&&\\& &2\ar@{..}[ur]&\\
& 1\ar[ur]^a&&3\ar@<.5mm>[ll]^d\ar@<-.5mm>[ll]_c\\&&2\ar[ur]_b&\\ \cov{\widetilde{\Lambda}_1}{\ZZ}= &1\ar[ur]^a&&3\ar@<.5mm>[ll]^d\ar@<-.5mm>[ll]_c\\& &2\ar[ur]_b&\\ &1\ar[ur]^a&&3\ar@<.5mm>[ll]^d\ar@<-.5mm>[ll]_c\\& &\ar@{..}[ur]&}\quad
\xymatrix@-.5cm{&&&\\ &&2\ar@{..}[ul]\ar[dr]^b&\\ &1\ar[dr]_a&&3\ar@<.5mm>[ll]^d\ar@<-.5mm>[ll]_c\\ &&2\ar[dr]^b&\\ \cov{\widetilde{\Lambda}_2}{\ZZ}=& 1\ar[dr]_a&&3\ar@<.5mm>[ll]^d\ar@<-.5mm>[ll]_c\\ &&2\ar[dr]^b&\\& 1&&3\ar@<.5mm>[ll]^d\ar@<-.5mm>[ll]_c\\& &\ar@{..}[ul]&} \quad
\xymatrix@-.5cm{ &&&\\&&2\ar[dr]_b&\\& 1\ar[ur]^a&&3\ar@/_3mm/@{..}[uull] \ar[ll]^d\\& &2\ar[dr]_b&\\ \cov{\widetilde{\Lambda}_3}{\ZZ}=&1\ar[ur]^a&&3\ar@/_3mm/[uull]_c\ar[ll]^d \\&&2\ar[dr]_b&\\& 1\ar[ur]^a&&3\ar@/_3mm/[uull]_c\ar[ll]^d\\&&&\\&&&\ar@/_3mm/@{..}[uull]}$$
It is then clear that the first two locally finite categories are equivalent, but not to the third one.
\end{exa}

Theorem~\ref{derivedeq1} can be generalized to the case that $\widetilde{\Lambda}_2$ is isomorphic not necessarily to $\widetilde{\Lambda}_1$, but to the endomorphism algebra of some cluster-tilting object in $\Cc_1$.

\begin{thma} \label{derivedeq1mutation}
Let $\Lambda_1$ and $\Lambda_2$ be two $\tau_2$-finite algebras of global dimension~$\leq 2$. Assume there is $T \in \Dd_1$ such that $\pi_1(T)$ is basic cluster tilting in $\Cc_1$, and
\begin{enumerate}
\item there is an isomorphism $\xymatrix{\End_{\Cc_1}(\pi_1 T)\ar[r]^-\sim & \End_{\Cc_2}(\pi_2\Lambda_2)}$
\item this isomorphism can be chosen in such a way that the two $\ZZ$-gradings defined on $\widetilde{\Lambda}_2$, given respectively by
$$\bigoplus_{q\in \ZZ}\Hom_{\Dd_2}(\Lambda_2,\SSS^{-q}\Lambda_2)\textrm{ and } \bigoplus_{p\in \ZZ}\Hom_{\Dd_1}(T, \SSS^{-p}T),$$
are equivalent.
\end{enumerate}
Then the algebras $\Lambda_1$ and $\Lambda_2$ are derived equivalent, and hence cluster equivalent.
\end{thma}

\begin{proof}
The object $\pi_1 T$ is a cluster-tilting object in $\Cc_1$. Hence the subcategory $\pi_1^{-1}(\pi_1 T)$ is a cluster-tilting subcategory of $\Dd_1$. It is immediate to see that
$$\pi_1^{-1}(\pi_1 T)=\add\{\SSS^p T \mid p\in \ZZ\}.$$
With the same argument used in the proof $(2)\Rightarrow (1)$ of Theorem~\ref{derivedeq1}, we show that Condition~(2) is equivalent to the fact that we have an equivalence of cluster-tilting subcategories with $\SSS$-action:
$$\Uu_1:=\add\{\SSS^p T \mid p\in \ZZ\}\simeq \add\{ \SSS^q \Lambda_2, q\in\ZZ\}=:\Uu_2$$
Then we can apply Theorem~\ref{clustertilt} to get the result.
\end{proof}

\begin{rema}
In Section~\ref{section_left_mutation} we will introduce the notion of mutation of graded quivers with potential, which makes it possible to check the assumptions of Theorem~\ref{derivedeq1mutation} more easily. Therefore we give an example there (Example~\ref{examplemutation}).
\end{rema}

\subsection{Classification of tilting complexes}

In this subsection we classify tilting complexes giving rise to derived equivalent algebras with the same canonical cluster tilting object.

Let $\Lambda$ be a $\tau_2$-finite algebra of global dimension~$\leq 2$. Let $\Lambda=P_1\oplus \cdots \oplus P_n$ the
decomposition of the free module $\Lambda$ into indecomposable
projectives. Let $T=\bigoplus_{i=1}^n\SSS^{-d_i}P_i$ be a
lift of the canonical cluster-tilting object $\pi(\Lambda)$ (cf.\ Example~\ref{listexamples}(3)). Our aim is to
determine when $T$ is a tilting complex with $\gldim(\End_\Dd(T))\leq 2$.

First recall a result from \cite{IO}.

\begin{prop}[{\cite[Theorem~4.5 and Proposition~4.7]{IO}}]\label{2apr}
Let $\Lambda$ be an algebra of global dimension
$\leq 2$. Let $\Lambda=P_0\oplus P_R$ be a decomposition such that
\begin{enumerate}
\item
$\Hom_\Dd(P_R,P_0)=0$, and
\item $\Ext^1_\Lambda(\mathbb{S}P_R,P_0)=0$ (recall that $\mathbb{S}$ is the Serre functor, so $\mathbb{S} P_R$ is the injective module corresponding to the projective module $P_R$).
\end{enumerate}
 Then the
complex $T=\SSS^{-1} P_0\oplus P_R$ is a tilting complex with $\gldim
\End_\Dd(T)\leq 2$. 
\end{prop}
In this case the complex $T$ is called a \emph{2-APR-tilt} of $\Lambda$.

We denote by $\Uu$ the cluster-tilting subcategory
$\add\{\SSS^p\Lambda \mid p\in \mathbb{Z}\}$ of $\Dd^b(\Lambda)$.
\begin{dfa}
An object $\Sigma$ in $\Uu$ is called a \emph{slice} if the following
holds:
\begin{enumerate}
\item $\Sigma$ intersects every $\SSS$-orbit in exactly one point.
\item $\Hom_\Dd(\Sigma,\SSS^p\Sigma)=0$ for all $p>0$.
\item $\Hom_\Dd(\Sigma, \SSS^p\Sigma[-1])=0$ for all $p\geq 0$.
\end{enumerate}
\end{dfa}
Note that, if $\Sigma$ is a slice,
then $\SSS \Sigma$ and $\SSS^{-1}\Sigma$ are also slices.

We define a partial order on slices. Let $\Sigma=\bigoplus_{i=1}^n\SSS^{-s_i}P_i$ and $\Sigma'=\bigoplus_{i=1}^n \SSS^{-t_i}P_i$ be two slices. Then we write $\Sigma\leq \Sigma'$ if for all $i=1,\ldots ,n$ we have $s_i\leq t_i$.  

For two complexes $\Sigma=\bigoplus_{i=1}^n\SSS^{-s_i}P_i$ and $\Sigma'=\bigoplus_{i=1}^n \SSS^{-t_i}P_i$, we will denote by $\max(\Sigma,\Sigma')$ the complex $\bigoplus_{i=1}^n\SSS^{-u_i}P_i$ where $u_i=\max(s_i,t_i)$, and by $\min(\Sigma,\Sigma')$ the complex $\bigoplus_{i=1}^n\SSS^{-v_i}P_i$ where $v_i=\min(s_i,t_i)$.

\begin{lema}\label{max}
If $\Sigma$ and $\Sigma'$ are slices, then $\max(\Sigma,\Sigma')$ and $\min(\Sigma,\Sigma')$ are slices.
\end{lema}

\begin{proof}
By definition $\max(\Sigma,\Sigma')$ intersects every $\SSS$-orbit in exactly one point. Let $S$ and $S'$ be indecomposable summands of $\max(\Sigma,\Sigma')$. We can assume that $S$ is a summand of $\Sigma$ and $S'$ a summand of $\Sigma'$. Then there exists $d\geq 0$ such that $\SSS^dS$ is a summand of $\Sigma'$.
For $p\geq 0$, the space $\Hom_\Dd(S,\SSS^p S')\simeq \Hom(\SSS^d S,\SSS^{p+d}S')$ vanishes since $\SSS^dS$ and $S'$ are summand of $\Sigma'$ and $p+d\geq 0$.
For the same reasons, the space $\Hom_\Dd(S,\SSS^p S'[-1])\simeq \Hom(\SSS^d S,\SSS^{p+d}S'[-1])$ vanishes. Hence $\max(\Sigma,\Sigma')$ is a slice. 

The proof is similar for $\min(\Sigma,\Sigma')$.
\end{proof}

\begin{lema}\label{lambda}
The object $\Lambda$ is a slice.
\end{lema}

\begin{proof}
Since the global dimension of $\Lambda$ is $\leq 2$, it is not hard to see (cf.\ \cite[Lemma~4.6]{Ami3}), that the cohomology of $\SSS^p \Lambda$ is in degree $\geq 2$ for $p\geq 1$. Therefore we immediately get (2) and (3).
\end{proof}

\begin{lema}\label{2aprslice}
Let $\Sigma$ and $\Sigma'$ be two slices such that $\SSS\Sigma'
\leq\Sigma\leq \Sigma'$. Assume that $\Sigma$ is a tilting complex with endomorphism algebra of global dimension~$\leq 2$. Then $\Sigma'$ is a $2$-APR-tilt of $\Sigma$.
\end{lema}

\begin{proof}
The slices $\Sigma'$ and $\SSS\Sigma'$ have no common summands, and
since $\SSS\Sigma'
\leq\Sigma\leq \Sigma'$, the slice $\Sigma$ is a direct summand of $\SSS\Sigma'\oplus\Sigma'$. 

Let $P_0$ be the intersection $\Sigma\cap\SSS\Sigma'$ and $P_R$ be the intersection $\Sigma\cap\Sigma'$.  Then we have $\SSS^{-1}P_0\oplus P_R=\Sigma'$.

We shall prove that the decomposition $\Sigma=P_0\oplus P_R$ satisfies the properties (1) and (2) of Theorem~\ref{2apr}. The space $\Hom_\Dd(P_R,P_0)$ vanishes since $P_R$ is in $\Sigma'$ and $P_0$ is in $\SSS \Sigma'$. The space $\Ext^{-1}_\Lambda(P_R,\SSS^{-1}P_0)$ vanishes since $P_R$ and $\SSS^{-1}P_0$ are both in $\Sigma'$. Therefore $\Sigma'$ is a 2-APR-tilt of $\Sigma$.
\end{proof}

\begin{prop}\label{2aprslice2}
All slices are iterated 2-APR-tilts of $\Lambda$.
\end{prop}

\begin{proof}
Let $\Sigma$ be a slice. Let $N$ be minimal with $\SSS^N \Lambda \leq \Sigma$, and $M$ maximal with $\Sigma \leq \SSS^M \Lambda$. We prove the claim by induction on $N - M$. For $N=M$ we have $\Sigma = \SSS^N \Lambda$, thus $\Sigma$ is an iterated $2$-APR-tilt of $\Lambda$. So assume $N > M$. Let $\Sigma' = \max(\SSS^{N-1} \Lambda, \Sigma)$. Then $\SSS \Sigma' \leq \Sigma \leq \Sigma'$, and by Lemma~\ref{2aprslice} $\Sigma$ is a $2$-APR-tilt of $\Sigma'$. Since $\SSS^{N-1} \leq \Sigma' \leq \SSS^M \Lambda$ we know inductively that $\Sigma'$ is an iterated $2$-APR-tilt of $\Lambda$, so $\Sigma$ is also an iterated $2$-APR-tilt of $\Lambda$.
\end{proof}

\begin{thma}\label{characterizationinverseimage}
For $T=\bigoplus_{i=1}^n\SSS^{-d_i}P_i$, the following are equivalent:
\begin{enumerate}
\item $T$ is a tilting complex with $\gldim \End(T)\leq 2$;
\item $T$ is a slice;
\item $T$ is an iterated $2$-APR-tilt of $\Lambda$;
\item  \begin{enumerate}
\item $\Ext^1_\Lambda(S_j,S_i)\neq 0$ or $\Ext^2_\Lambda(S_i,S_j)\neq 0$
  implies that $d_j-d_i\in\{0,1\}$ and
\item for any $r$ we have $\Hom_\Dd(P_i,\SSS^{-r}P_j[-1])\neq 0$
  implies that $d_j-d_i<r$.
\end{enumerate}
\end{enumerate}
\end{thma}

\begin{proof}
We have $(2)\Rightarrow (3)\Rightarrow (1)$ by Proposition~\ref{2aprslice2} and by Theorem~\ref{2apr}.

Assume now that $T$ is a tilting complex whose endomorphism algebra $\Gamma=\End_{\Dd^b(\Lambda)}(T)$ is of global dimension~$\leq 2$. Then by Lemma~\ref{lambda}, $\Gamma$ is a slice in $\Dd^b(\Gamma)$. Since $T$ is a tilting complex, there is a triangle equivalence $$\xymatrix{\RHom_{\Dd^b(\Lambda)}(T,-) \colon \Dd^b(\Lambda)\ar[r]^-\sim & \Dd^b(\Gamma)}.$$ By uniqueness of the Serre functor, the functors $_\Lambda\SSS$ and $_\Gamma\SSS$ are isomorphic. Since $T$ is of the form $\bigoplus_{i=1}^n\SSS^{-d_i}P_i$, we have the following commutative square 
$$\xymatrix{\Dd^b(\Lambda)\ar[r]^\sim & \Dd(\Gamma)\\ \Uu_\Lambda\ar[r]^\sim \ar@{^(->}[u]& \Uu_\Gamma\ar@{^(->}[u]}$$
where $\Uu_\Lambda$ is the additive subcategory $\add\{\SSS^p \Lambda \mid p\in\mathbb{Z}\}$ of $\Dd^b(\Lambda)$.
Therefore $\Sigma$ is a slice in $\Dd^b(\Lambda)$ if and only if $\RHom_{\Dd^b(\Lambda)}(T,\Sigma)$ is a slice in $\Dd^b(\Gamma)$. Hence $T$ is a slice in $\Dd^b(\Lambda)$ and we have $(1)\Rightarrow (2)$.

For the last equivalence, first note that Condition~(4b) is clearly equivalent to Condition~(3) of the definition of a slice.

Let $\widetilde{\Lambda}$ be the endomorphism algebra $\End_{\Cc}(\pi \Lambda)=\bigoplus_{p\in \mathbb{Z}}\Hom_{\Dd^b(\Lambda)}(\Lambda,\SSS^{-p}\Lambda)$. This algebra is a graded algebra with positive grading generated in degrees $0$ and $1$. The arrows of its quiver $Q_{\widetilde{\Lambda}}$ have degree $0$ or $1$. Moreover we have 
$$\sharp\{i\rightarrow j \in Q_{\widetilde{\Lambda}}\}= \dim\Ext_\Lambda^1(S_j,S_i)+\dim \Ext^2_\Lambda(S_i,S_j),$$
where $S_i$ (resp.\ $S_j$) is the simple module top of $P_i$ (resp.\ $P_j$). 

Let $T$ be a tilting complex of the form $\bigoplus_{i=1}^n\SSS^{-d_i}P_i=\bigoplus_{i=1}^nT_i$ and such that its endomorphism algebra $\Gamma=\End_{\Dd^b(\Lambda)}(T)$ is of global dimension~$\leq 2$. The algebra  $$\widetilde{\Gamma}=\End_{\Cc}(\pi(T))=\bigoplus_{p\in\mathbb{Z}}\Hom_{\Dd^b(\Lambda)}(T,\SSS^{-p}T)$$ is isomorphic to $\widetilde{\Lambda}$.
By the above remark, every arrow of the graded quiver $Q_{\widetilde{\Gamma}}$ has degree $0$ or $1$.  
Assume that $\Ext^1_\Lambda(S_j, S_i)$ does not vanish. This means that there exists an irreducible map $P_i\rightarrow P_j$, therefore there exists an irreducible map $T_i\rightarrow \SSS^{-d_i+d_j}T_j$. But since every arrow of the graded quiver $Q_{\widetilde{\Gamma}}$ has degree $0$ or $1$, we have $d_i-d_j\in\{0,1\}$.
Assume now that $\Ext^2_\Lambda(S_i,S_j)$ does not vanish. This means that there exists an irreducible map $P_i\rightarrow \SSS^{-1}P_j$, therefore there exists an irreducible map $T_i\rightarrow \SSS^{-d_i+d_j-1}T_j$. Since every arrow of the graded quiver $Q_{\widetilde{\Gamma}}$ has degree $0$ or $1$, we have $d_i-d_j-1\in\{0,1\}$. Hence we have $d_j-d_i\in\{0,1\}$. Therefore we have $(1)\Rightarrow (4)$.

Now assume that we have Condition~(4) for the object $T=\bigoplus_{i=1}^n\SSS^{-d_i}P_i$. We will prove that it is a slice. We obviously have Condition~(1) and we have Condition~(3) since we have (4b). The algebra
\[ \widetilde{\Gamma}=\End_{\Cc}(\pi(T))=\bigoplus_{p\in\mathbb{Z}}\Hom_{\Dd^b(\Lambda)}(T,\SSS^{-p}T) \]
is isomorphic to $\widetilde{\Lambda}$. Condition~(4a) exactly means, as we just saw above, that the graded quiver $Q_{\widetilde{\Gamma}}$ only has arrows of degree $0$ and $1$. A non-zero morphism in $\Hom_{\Dd^b(\Lambda)}(T,\SSS^pT)$ implies that there exists a path of degree $-p$ in the graded quiver $Q_{\widetilde{\Gamma}}$. Hence $p$ must be non-positive, and we have Condition~(2) of the definition of a slice. Therefore we have $(4)\Rightarrow (2)$.
\end{proof}

\section{Left (and right) mutation in the derived category and graded quivers with potential} \label{section_left_mutation}

The aim of this section is to provide a combinatorial description of the mutation in the derived category. We first give a link between right and left mutation. Then we recall some results about quivers with potential (QP for short) and generalized cluster categories attached to them. Finally, we define the notion of mutation of a graded QP and state the main result which permits to compute explicitly the grading of the endomorphism ring of a cluster-tilting object in the cluster category.

\subsection{Relation between left mutation and right mutation}

Let $\Tt$ be a triangulated category with Serre functor $\mathbb{S}$ and $\Vv\subset\Tt$ be a cluster-tilting subcategory. Let $X\in \Vv$ be indecomposable, and set
\[ \Vv' := \add ( \ind(\Vv) \setminus \{\SSS^p X \mid p\in\ZZ\}), \]
where $\ind(\Vv)$ denotes the indecomposable objects in $\Vv$. Then by Theorem~\ref{IyamaYoshino} there exists a unique cluster-tilting subcategory $\Vv^*$ with $\Vv' \subseteq \Vv^* \neq \Vv$.
\begin{prop} \label{tirtil}
In the setup of Theorem~\ref{IyamaYoshino} (left and right exchange triangles), if for any $p \neq 0$ any map $X\rightarrow \SSS^pX$ factors through $\Vv'$, then we have $$X^R\simeq \SSS X^L.$$
\end{prop}

\begin{proof}
Let $\xymatrix{X^L[-1]\ar[r]^-f & X\ar[r]^-u & B\ar[r] & X^L}$ be the left approximation triangle. Let $g \colon X^L[-1]\rightarrow U$ be a non-zero morphism with $U$ an indecomposable object in $\Vv$. Then $U$ is not in $\Vv'$ since $\Hom(X^L[-1], \Vv') = 0$ (Theorem~\ref{IyamaYoshino}). Hence there exists a $p$ such that $U=\SSS^pX$. Since $f$ is a left $\Vv$-approximation $g$ factors through $f$. Since for $p \neq 0$ all maps $X \rightarrow \SSS^pX$ factor through $\Vv'$, hence through $u$, we have $p=0$. Therefore we have
\[ \Hom_{\Tt}(X^L[-1],\Vv)=\Hom_{\Tt}(X^L[-1],X). \]
But now we have isomorphisms in $\mod \Vv$
\[ D\Hom_{\Tt}(\Vv,\SSS X^L[1])\simeq\Hom_{\Tt}(X^L[-1],\Vv)=\Hom_{\Tt}(X^L[-1],X)\simeq D\Hom_{\Tt}(X,\SSS X^L[1]). \]
Hence the only indecomposable object in $\Vv$ admitting non-zero maps to $\SSS X^L[1]$ is $X$.

Now let
\[ \xymatrix{\SSS X^L\ar[r] & B' \ar[r] & X' \ar[r] & \SSS X^L[1]} \]
be the triangle obtained from a left $\Vv'$-approximation of $\SSS X^L$. By Theorem~\ref{IyamaYoshino} we have $X' = \SSS^q X$ for some $q \in \mathbb{Z}$. By the above observation we have $q = 0$, so $X' = X$.

Now, since the above map $\xymatrix@-.3cm{B' \ar[r] & X' = X}$ is a right $\Vv'$-approximation, this triangle is precisely the triangle defining $X^R$. So we see $X^R = \SSS X^L$.
\end{proof}

\subsection{Jacobian algebras and cluster-tilting objects}\label{subsection jacobian}

Quivers with potentials and the associated Jacobian algebras have been studied in \cite{DWZ}.
Let $Q$ be a finite quiver. For each arrow $a$ in $Q$, the \emph{cyclic derivative} $\partial_a$ with
respect to $a$ is the unique linear map
$$\partial_a \colon kQ\rightarrow kQ$$
which sends a path $p$ to the sum $\sum_{p=uav}vu$ taken over all decompositions of the 
path $p$ (where $u$ and $v$ are possibly idempotent elements $e_i$ associated to  a vertex $i$).
A \emph{potential} on $Q$ is any (possibly infinite) linear combination $W$ of cycles 
in $Q$. The associated Jacobian algebra is
$$\Jac(Q,W):=k\hat{Q}/\langle \partial_aW; a\in Q_1\rangle,$$
where $k\hat{Q}$ is the completed path algebra, that is the completion of $kQ$ with respect at the ideal generated by the arrows, and $\langle \partial_aW; a\in Q_1\rangle$ is the closure of the ideal generated by $\partial_aW$ for $a\in Q_1$.

Associated with any quiver with potential $(Q,W)$, a cluster category $\Cc_{(Q,W)}$ is constructed in \cite{Ami3}. It uses the notion of Ginzburg dg algebra. We refer the reader to \cite{Ami3} for an explicit construction. When the associated Jacobian algebra is finite dimensional, the category $\Cc_{(Q,W)}$ is $2$-Calabi-Yau and endowed with a canonical cluster-tilting object $T_{(Q,W)}$ (that is an object such that $\add (T_{(Q,W)})$ is a cluster-tilting subcategory) whose endomorphism algebra is isomorphic to $\Jac(Q,W)$. The next result gives a link between cluster categories associated with algebras of global dimension~$\leq 2$ and cluster categories associated with QP.

\begin{thma}[{\cite[Theorem 6.11 \textit{a)}]{Kel10}}] \label{keller}
Let $\Lambda=kQ/I$ be a $\tau_2$-finite algebra of global dimension~$\leq 2$, such that $I$
is generated by a finite minimal set of relations $\{ r_i \}$. (By this we mean that the set $\{r_i\}$ is the disjoint union of sets representing a basis of the $\Ext_{\Lambda}^2$-space between any two simple $\Lambda$-modules.) The relation
$r_i$ starts at the vertex $s(r_i)$ and ends at the vertex
$t(r_i)$. 
Then there is a triangle equivalence:
$$ \Cc_\Lambda \simeq \Cc_{(\widetilde{Q},W)},$$ where the quiver $\widetilde{Q}$ is the quiver $Q$ with additional
arrows $a_i \colon t(r_i)\rightarrow s(r_i)$, and the potential $W$ is
$\sum_i a_ir_i$. This equivalence sends the cluster-tilting object $\pi(\Lambda)$ on the cluster-tilting object $T_{(\widetilde{Q},W)}$.

As a consequence we have an isomorphism of algebras:
$$\End_\Cc(\pi \Lambda) \simeq \Jac(\widetilde{Q},W).$$
\end{thma}

Here is an immediate observation which introduces the natural grading of $\End_\Cc(\pi \Lambda)$ to Theorem~\ref{keller}. It gives a motivation for introducing the left graded-mutation of a quiver.

\begin{prop}\label{Zgrading}
Let $\Lambda=kQ/I$ be a $\tau_2$-finite algebra of global dimension~$\leq 2$. Denote by $(\widetilde{Q}, W)$ the quiver with potential defined in Theorem~\ref{keller}. 
Then there exists a unique $\ZZ$-grading on $\widetilde{Q}$ such that
\begin{enumerate}
\item the potential $W$ is homogeneous of degree $1$;
\item there is an isomorphism of quivers $\xymatrix{\widetilde{Q}^{\{0\}}\ar[r]^-\sim & Q}$, where $\widetilde{Q}^{\{0\}}$ is the subquiver of $\widetilde{Q}$ of arrows of degree $0$.
\end{enumerate}
This grading on $\widetilde{Q}$ yields a grading on $\Jac(\widetilde{Q},W)$ and we have an isomorphism of $\ZZ$-graded algebras
$$\xymatrix{\Jac(\widetilde{Q},W)\ar[r]^-\sim_-{\ZZ} &\bigoplus_{p\in\ZZ}\Hom_{\Dd^\ZZ}(\Lambda,\SSS^{-p}\Lambda)}.$$
\end{prop}

\begin{proof}
This is achieved by giving the arrows in $Q_1$ degree zero, and the arrows in $\widetilde{Q}_1 \setminus Q_1$ (that is the arrows corresponding to minimal relations) degree one.
\end{proof}

\subsection{Left (and right) mutation of a graded quiver with potential}

Extending Fomin and Zelevinsky mutations of quivers \cite{FZ1}, Derksen, Weyman, and Zelevinsky have introduced the notion of mutation of quivers with potential in \cite{DWZ}. We adapt this notion to $G$-graded quivers with  potential homogeneous of degree $r\in G$. In the following subsections of this section we will use this definition for $G=\ZZ$ and $r=1$, and in Section~\ref{section_gr_der_eq} for $G=\ZZ^2$ and $r=(1,1)$.

\begin{dfa}
Let $(Q,W,d)$ be a $G$-graded quiver with potential homogeneous
  of degree $r$ ($G$-graded QP for short). Let $i\in Q_0$ be a vertex, such that there are neither loops
  nor 2-cycles incident to $i$. We define $\tilde{\mu}_i^L(Q,W,d)=(Q',W',d')$. The quiver $Q'$ is defined as in
  \cite{DWZ} as follows
\begin{itemize}
\item for any subquiver $\xymatrix{u\ar[r]^a &i\ar[r]^b & v}$ with
  $i$, $u$ and $v$ pairwise different vertices, we add an arrow $[ba] \colon u\rightarrow v$;
\item we replace all arrows $a$ incident with $i$ by an arrow $a^*$ in
  the opposite direction.
\end{itemize}

The potential $W'$ is also defined as in \cite{DWZ} by the sum $[W]+W^*$ where $[W]$ is formed
from the
potential $W$ replacing all compositions $ba$ through the vertex $i$ by the new arrows $[ba]$, and where $W^*$ is the sum $\sum b^*a^*[ba]$.

The new degree $d'$ is defined as follows:

\begin{itemize}
\item $d'(a)=d(a)$ if $a$ is an arrow of $Q$ and $Q'$;
\item $d'([ba])=d(b)+d(a)$ if $ba$ is a composition passing through
  $i$;
\item $d'(a^*)=-d(a)+r$ if the target of $a$ is $i$;
\item $d'(b^*)=-d(b)$ if the source of $b$ is $i$.
\end{itemize} 
\end{dfa}

Similarly, we can define $\tilde{\mu}_i^R(Q,W,d)=(Q',W',d')$ by setting
$d'(a^*)=-d(a)$ for arrows $a$ such that $t(a)=i$, and
$d'(b^*)=-d(b)+r$ for arrows $b$ with $s(b)=i$.

As in \cite{DWZ}, it is possible to define trivial and reduced graded quivers with potential. A $G$-graded QP $(Q,W,d)$ is \emph{trivial} if the potential $W$ is in the space $kQ_2$ spanned by paths of length 2, and if 
the Jacobian algebra $\Jac(Q,W,d)$ is isomorphic to the semisimple algebra $kQ_0$. A $G$-graded QP $(Q,W,d)$ is \emph{reduced} if $W\cap kQ_2$ is zero.

\begin{dfa}
Two $G$-graded QP $(Q,W,d)$ and $(Q',W',d')$ are \emph{graded right equivalent} if there exists an isomorphism of $G$-graded algebras $\xymatrix{\varphi \colon (kQ,d)\ar[r]^\sim_{G} & (kQ',d')}$ such that $\varphi|_{kQ_0} = {\rm id}$ and such that $\varphi(W)$ is cyclically equivalent to $W'$ in the sense of \cite{DWZ}.
\end{dfa}

Then Splitting Theorem of \cite{DWZ} still holds in the graded case. Indeed the right equivalence constructed in the proof of \cite[Lemma~4.8]{DWZ}, is graded.

\begin{thma}[{Graded Splitting Theorem -- compare \cite[Theorem~4.6]{DWZ}}]
Let $(Q,W,d)$ be a $G$-graded QP. Then there exists a (unique up to graded right equivalence) reduced graded QP $(Q^{\rm red},W^{\rm red},d^{\rm red})$ and a (unique up to graded right equivalence) trivial QP $(Q^{\rm triv},W^{\rm triv},d^{\rm triv})$ such that $(Q,W,d)$ is graded right equivalent to the direct sum $$(Q^{\rm red},W^{\rm red},d^{\rm red})\oplus(Q^{\rm triv},W^{\rm triv},d^{\rm triv}).$$
(The direct sum means that the vertices of both summands and the sum coincide, and the arrows of the sum are the disjoint union if the arrows of the summands.)
\end{thma}

Therefore, we can deduce the following.

\begin{prop}
Let $(Q,W,d)$ be a $G$-graded quiver with potential homogeneous of degree
$r$ and let $i$ be vertex of $Q$ without incident loops or
$2$-cycles. Then the reduction $(Q'^{\rm red},W'^{\rm red}, d')$ in the sense of \cite{DWZ} of  $\tilde{\mu}_i^L(Q,W,d)=(Q',W',d')$ has potential homogeneous of degree $r$ with
respect to the grading $d'$.
\end{prop}

\begin{proof}
It is an easy computation to check that the potential $W'$ is
homogeneous of degree $r$ with respect to $d'$. Then the graded splitting theorem implies that the reduction process does not change the homogeneity of the potential. 
\end{proof}
\begin{dfa} The Graded Splitting Theorem and the above proposition allows us to defined \emph{the left mutation at vertex $i$} $\mu_i^L(Q,W,d)$ of a $G$-graded quiver with potential $(Q,W,d)$ as the reduction of the graded quiver with potential $\tilde{\mu}_i^L(Q,W,d)$. Similarly we define \emph{the right mutation at vertex $i$} $\mu_i^R(Q,W,d)$ of a graded QP $(Q,W,d)$ as the reduction of the graded quiver with potential $\tilde{\mu}_i^R(Q,W,d)$.
\end{dfa}

One immediately checks the following.

\begin{lema}\label{lemma left=right inverse}
Let $(Q,W,d)$ be a $G$-graded QP with potential homogeneous of degree $r$, and let $i$ be a vertex without incident loops or $2$-cycles. Then the $G$-graded QP $(Q,W,d)$, $\mu_i^L\mu_i^R(Q,W,d)$ and $\mu_i^R\mu_i^L(Q,W,d)$ are graded right equivalent.
\end{lema}

The following lemma gives a direct link between left mutation and right mutation of a graded quiver.
\begin{lema}
Let $(Q,W,d)$ be a $G$-graded quiver with potential homogeneous of degree
$r$ and let $i$ be vertex of $Q$ without incident loops or
$2$-cycles. Then there is a graded equivalence
$$\mod \Jac(\mu_i^L(Q,W,d))\rightarrow \mod \Jac(\mu_i^R(Q,W,d)).$$
\end{lema}

\begin{proof}
The algebras $\Jac(\mu_i^L(Q,W,d))$ and $\Jac(\mu_i^R(Q,W,d))$ are isomorphic.
Let us decompose the Jacobian algebra $\Jac(\mu_i^L(Q,W,d))$ as a direct sum of graded projective $\Jac(\mu_i^L(Q,W,d))$-modules.
\[ \Jac(\mu_i^L(Q,W,d))= \bigoplus_{j\in Q_0}P_j.\]
One can check that the graded endomorphism algebra of
\[ (\bigoplus_{j\in Q_0, j\neq i}P_j ) \oplus P_i\langle -r\rangle \]
is then isomorphic as graded algebra to $\Jac(\mu_i^R(Q,W,d))$.
\end{proof}

\subsection{Mutation and generalized cluster categories associated with QP}

Let $(Q,W)$ be a quiver with potential, such that the Jacobian algebra $\Jac(Q,W)$ is finite dimensional, and $T_{(Q,W)}\in\Cc_{(Q,W)}$ be the canonical cluster-tilting object of the generalized cluster category $\Cc_{(Q,W)}$. There is a canonical bijection between the vertices of $Q$ and the indecomposable objects of $T_{(Q,W)}\simeq T_i\oplus T'$. Let $i$ be in $Q_0$, $T_i$ be the corresponding summand of $T_{(Q,W)}$, and 
\[\xymatrix{T_i\ar[r] & B\ar[r] & T_i^L\ar[r] & T_i[1]} \quad \textrm{and }\xymatrix{T_i^R\ar[r] & B'\ar[r] & T_i\ar[r] & T_i^R[1]}\] be the approximation triangles as defined in Theorem~\ref{IyamaYoshino}. Then, we have $T_i^R\simeq T_i^L$ since the category $\Cc_{(Q,W)}$ is $2$-Calabi-Yau. We denote by $\mu_{i}(T_{(Q,W)})$ the new cluster-tilting object $T'\oplus T_i^L$. The following fundamental result links the \cite{DWZ}-mutation of QP and the \cite{IY}-mutation of the cluster-tilting object $T_{(Q,W)}$. 

\begin{thma}[Keller-Yang \cite{KY}] \label{kelleryang}
Let $(Q,W)$ be a quiver with potential whose Jacobian algebra is finite dimensional, and $i\in Q_0$ a vertex such that there is no 2-cycles nor loops at vertex $i$ in $Q$. Then there exists a triangle equivalence $$\Cc_{\mu_i(Q,W)}\simeq \Cc_{(Q,W)}$$ sending the cluster-tilting object $T_{\mu_i(Q,W)}\in \Cc_{\mu_i(Q,W)}$ on the cluster-tilting object $\mu_{i}(T_{(Q,W)})\in \Cc_{(Q,W)}$, where $T_i$ is the indecomposable summand of $T_{(Q,W)}$ associated with the vertex $i$ of $Q$, and where $\mu_i(Q,W)$ is the mutation of the quiver with potential $(Q,W)$ at vertex $i$.
\end{thma}

This triangle equivalence is compatible with vertices in the following sense: the bijection between the vertices of $Q$ and the indecomposable direct summands of $T_{(Q,W)}$ together with this equivalence induce a bijection between the vertices of $\tilde{\mu}_i(Q,W)$ and the indecomposable summands of $T_{\mu_i(Q,W)}$, which is the canonical one.
\medskip

 The compatibility between \cite{DWZ}-mutation of quivers with potential and \cite{IY}-mutation of cluster tilting objects, given by Theorem~\ref{kelleryang}, can be understood more precisely in the following way:

Let $T_{(Q,W)} \simeq T_i\oplus T'$ as above. If the 
 map $T_i\to B$ is a minimal left $(\add T')$-approximation, then $B$ is isomorphic to $\bigoplus_{j, a \colon i\to j}T_j$. So for any arrow $a \colon i \to j$ in $Q$, there is a non-zero map $T^j \to T_i^L$. This map corresponds to the new arrow $a^* \colon j \to i$ in $\tilde{\mu}_i(Q,W)$. In the same way, for any arrow $b \colon j \to i$ in $Q$, there is a non-zero map $T_i^R \simeq T_i^L \to T_j$ which corresponds to the new arrow $b^*$ in $\tilde{\mu}_i(Q,W)$. Furthermore, arrows $[ba] \colon j \to \ell$ in $\mu_i(Q)$, where $a \colon j\to i$ and $b \colon i\to \ell$ are in $Q$, correspond to the composition of the maps associated to $b$ and $a$.

\subsection{Relation between mutation of graded QP and mutation in the derived category}

Let $\Lambda$ be a $\tau_2$-finite algebra of global dimension~$\leq 2$. Let $T \simeq T_1 \oplus \ldots \oplus T_n$ be an object in $\Dd^b(\Lambda)$ such that $\pi(T)$ is a (basic) cluster-tilting object in $\Cc_\Lambda$. By Proposition~\ref{preimage}, the category $\Uu_T = \add\{\SSS^pT \mid p\in \ZZ\}=\pi^{-1}(\pi(T))$ is a cluster-tilting subcategory of $\Dd^b(\Lambda)$. 

Let $i \in \{1, \ldots, n\}$, and $T' = \oplus_{j \neq i} T_j$.
We denote by $\Uu_i$ the additive subcategory $\add\{\SSS^p T' \mid p\in \mathbb{Z}\}$ of $\Dd^b(\Lambda)$. Consider the left exchange triangle associated to $T_i$ (defined in Theorem~\ref{IyamaYoshino}) in $\Dd^b(\Lambda)$
$$\xymatrix{T_i\ar[r]^u & B\ar[r]^v & T_i^L\ar[r]^w & T_i[1]}$$
where $u \colon T_i\rightarrow B$ is a minimal left $\Uu_i$-approximation of $T_i$. We denote by $\mu_i^L(T)$ the object $T'\oplus T_i^L\in\Dd^b(\Lambda)$. Then, by Theorem~\ref{IyamaYoshino}, the category $\add\{\SSS^p (\mu_i^L(T)) \mid p\in \ZZ\}$ is a cluster-tilting subcategory of $\Dd^b(\Lambda)$, and by Proposition~\ref{prop_image_ct_again} the object $\pi(\mu_i^L(T))$ is cluster-tilting in $\Cc_\Lambda$.

The endomorphism  algebras of $\pi(T)$ and $\pi(\mu_i^L(T))$ are naturally $\ZZ$-graded since we have 
$$\End_\Cc(\pi(T))\simeq \bigoplus_{p\in\ZZ}\Hom_{\Dd}(T,\SSS^pT)\textrm{ and } \End_\Cc(\pi(\mu_i^L(T)))\simeq \bigoplus_{p\in \mathbb{Z}}\Hom_\Dd(\mu_i^L(T),\SSS^{-p}(\mu_i^L(T))).$$  
  
 The following theorem links the gradings of $\End_{\Cc}(\pi (\Lambda))$ and $\End_\Cc(\pi(T))$, when $T$ is obtained from $\Lambda$ by iterated mutations and is mainly a consequence of Theorem~\ref{kelleryang}.

\begin{thma}\label{Zgradedbirs}
Let $\Lambda=kQ/I$ be a $\tau_2$-finite algebra of global dimension~$\leq 2$, and denote by $(\widetilde{Q},W,d)$ the graded QP defined in Proposition~\ref{Zgrading}. Assume that there exists a sequence $i_1,i_2,\ldots,i_l$ of vertices of $\widetilde{Q}$ such that for any $j=0,\ldots,l$ there is no $2$-cycle on the vertex $i_{j+1}$ in the quiver $Q^{j}$ where $(Q^j,W^j):=\mu_{i_{j}}\circ\cdots\circ \mu_{i_1}(\widetilde{Q},W)$. Denote by $T$ the object in $\Dd^b(\Lambda)$ defined by $T:=\mu^L_{i_l}\circ\cdots\circ\mu_{i_1}^L(\Lambda)$.

Then there is an isomorphism of $\ZZ$-graded algebras
$$\xymatrix{\bigoplus_{p\in\mathbb{Z}}\Hom_{\Dd}(T,\SSS^{-p}(T))\ar[r]^-\sim_-{\ZZ}& \Jac(\mu_{i_l}^L\circ\cdots\circ\mu_{i_1}^L(\widetilde{Q},W,d)).}$$
\end{thma}

\begin{proof}
For $j=1,\ldots,l$ denote by $T^{j}$ the object $\mu^L_{i_{j}}\circ\cdots\circ\mu_{i_1}^L(\Lambda)$, and by $(Q^{j},W^{j},d^{j})$ the graded QP $\mu_{i_{j}}^L\circ\cdots\circ\mu_{i_1}^L(\widetilde{Q},W,d)$. We put $T^0:=\Lambda$ and $(Q^0,W^0,d^0):=(\widetilde{Q},W,d)$. The object $\pi(T^{j})$ is a cluster-tilting object of the cluster category $\Cc_\Lambda$. By Theorem \ref{keller}, there is a triangle equivalence $\Cc_\Lambda\simeq \Cc_{(Q^0,W^0)}$ sending the cluster-tilting $\pi(\Lambda)$ on the cluster-tilting object $T_{(Q^0,W^0)}\in\Cc_{(Q^0,W^0)}$. 

Now note that the quiver $Q^0$ does not have loops. Indeed, since $\Lambda$ is of finite global dimension, its Gabriel quiver does not contain any loop \cite{Len,Igu}. Moreover since $\widetilde{\Lambda}$ is finite dimensional, it is easy to see that there is no loop of degree 1 in the quiver $Q^0$. Furthermore, if a QP does not contain any loops, after one mutation at a vertex not on a $2$-cycle, it still does not contain any loops. Therefore by an immediate induction, one checks that $(Q^j,W^j,d^j)$ does not have loops for any $j=0,\ldots, l$.

Hence we can apply iteratively Theorem~\ref{kelleryang} and we obtain an equivalence of triangulated categories $$\Cc_{(Q^0,W^0)}\simeq \Cc_{(Q^j,W^j)}$$
sending $\mu_{i_j}\circ\cdots \circ \mu_{i_1}(T_{(Q^0,W^0)})$ onto $T_{(Q^j,W^j)}$. Hence we obtain an isomorphism of algebras
$$\xymatrix{\bigoplus_{p\in\mathbb{Z}}\Hom_{\Dd}(T^j,\SSS^{-p}(T^j))\ar[r]^-\sim& \Jac(Q^{j},W^{j}).}$$
The only thing to check is that this isomorphism preserves the grading. 
We proceed by induction on the number of mutations $j$. For $j=0$, we have 
$$\bigoplus_{p\in\mathbb{Z}}\Hom_{\Dd}(T^0,\SSS^{-p}(T^0))\underset{\mathbb{Z}}{\simeq} \bigoplus_{p\in\mathbb{Z}}\Hom_\Dd(\Lambda,\SSS^{-p}(\Lambda)) \underset{\mathbb{Z}}{\simeq} \Jac(Q^{0},W^{0},d^0)$$ by Proposition~\ref{Zgrading}. So assume the result for $j\geq 0$, we have by induction hypothesis $$\xymatrix{\bigoplus_{p\in\mathbb{Z}}\Hom_{\Dd}(T^{j},\SSS^{-p}(T^{j}))\ar[r]^-\sim_-{\ZZ}& \Jac(Q^{j},W^{j},d^{j}).}$$

Let $(T^j)'$ be the direct sum of the summands of $T^j$ not corresponding to the vertex $i_{j+1}$. Then the left exchange triangle is of the form
\[ T_{i_{j+1}}^{j} \to \bigoplus_{a \colon i_{j+1} \to s} \SSS^{-d(a)} T^{j}_s \to (T_{i_{j+1}}^{j})^L \to T_{i_{j+1}}^{j}[1], \]
where the sum in the second term runs over all arrows of $Q^j$ starting in $i_{j+1}$. An arrow $a^* \colon s \to i_{j+1}$ of the mutated quiver corresponds to the component map $\SSS^{-d(a)} T^{j}_s \to (T_{i_{j+1}}^{j})^L$, and therefore has degree $-d(a)$.

Similarly we see that, for an arrow $b \colon s \to i_{j+1}$, the new arrow $b^*$ corresponds to map $(T_{i_{j+1}}^{j})^R \to \SSS^{d(b)} T^{j}_s$, and thus, by Proposition~\ref{tirtil}, to a map $\SSS (T_{i_{j+1}}^{j})^L \to \SSS^{d(b)} T^{j}_s$. Therefore the degree of $b^*$ is $1 - d(b)$.

Finally, any arrow of type $[ab]$ in the quiver of $\Jac(\mu^L_{i_{j+1}}(Q^{j},W^{j},d^{j}))$ corresponds to a non-zero composition $\xymatrix{\SSS^{d(b)}T^{j}_r\ar[r] & T^{j}_{i_{j+1}} \ar[r]& \SSS^{-d(a)}T^{j}_s}$, therefore its degree in $\End_{\Cc}(\pi(\mu_{i_{j+1}}^L(T^j))$ is $d(a)+d(b)$.

Thus using $T^{j+1}=\mu_{i_{j+1}}^L(T^j)$ we obtain an isomorphism of graded algebras
\[ \bigoplus_{p\in\ZZ}\Hom_{\Dd}(T^{j+1},\SSS^{-p}T^{j+1})\underset{\ZZ}{\simeq}\Jac (\mu_{i_{j+1}}^L(Q^j,W^j,d^j))=\Jac(Q^{j+1},W^{j+1},d^{j+1}). \qedhere \]

\end{proof}

\begin{rema}
There exists a similar `right version' of this theorem. 
\end{rema}

Graded quivers with potential and Theorem~\ref{Zgradedbirs} permit to see the existence of a $T$ as in Theorem~\ref{derivedeq1mutation} without explicitly constructing it, and thus to show that certain algebras are derived equivalent. More precisely we have the following consequence.
 \begin{cora}\label{graded mutation and derived equivalence}
Let $\Lambda_1$ and $\Lambda_2$ be two algebras of global dimension 2 , which are $\tau_2$-finite. Assume that one can pass from the graded QP  associated with $\widetilde{\Lambda}_1$ to the graded QP associated with $\widetilde{\Lambda}_2$ using a finite sequence of left and right mutations at vertices not on $2$-cycles.  Then the algebras $\Lambda_1$ and $\Lambda_2$ are derived equivalent.
\end{cora}

\begin{exa}\label{examplemutation}
 Let $\Lambda_1$ and $\Lambda_2$ be the following algebras: 
 \[\scalebox{.8}
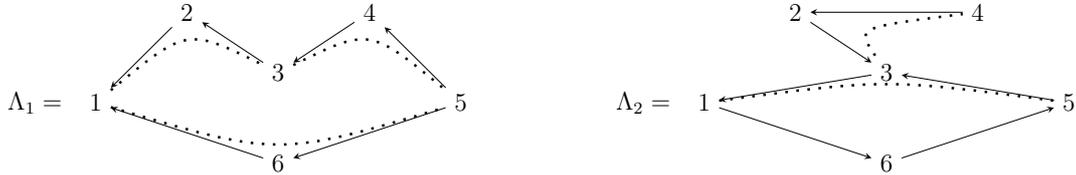
{
\begin{tikzpicture}[>=stealth,scale=.5]
\node at (-2,0){$\Lambda_1=$};
\node (P1) at (0,0){$1$};
\node (P2) at (3,3){$2$};
\node (P3) at (6,1){$3$};
\node (P4) at (9,3){$4$};
\node (P5) at (12,0){$5$};
\node (P6) at (6,-2){$6$};
\node at (18,0){$\Lambda_2=$};
\node (Q1) at (20,0){$1$};
\node (Q2) at (23,3){$2$};
\node (Q3) at (26,1){$3$};
\node (Q4) at (29,3){$4$};
\node (Q5) at (32,0){$5$};
\node (Q6) at (26,-2){$6$};
\draw [->] (P2) -- (P1);
\draw [->] (P3) -- (P2);
\draw [->] (P4) -- (P3);
\draw [->] (P5) -- (P4);
\draw [->] (P5) -- (P6);
\draw [->] (P6) -- (P1);
 \draw [loosely dotted,very thick] (P1) .. controls (3,2.5) .. (P3);
\draw [loosely dotted,very thick] (P3) .. controls (9,2.5) .. (P5);
\draw [loosely dotted,very thick] (P1) .. controls (6,-1.8) .. (P5);
\draw [->] (Q3) -- (Q1);
\draw [->] (Q2) -- (Q3);
\draw [->] (Q4) -- (Q2);
\draw [->] (Q5) -- (Q3);
\draw [->] (Q6) -- (Q5);
\draw [->] (Q1) -- (Q6);
\draw [loosely dotted,very thick] (Q3) .. controls (25,2.5) .. (Q4);
\draw [loosely dotted,very thick] (Q1) .. controls (26,0.8) .. (Q5);
\end{tikzpicture}}
\]
One can easily compute the quiver with potential associated with the algebras $\widetilde{\Lambda}_1$ and $\widetilde{\Lambda}_2$. Their graded quivers are
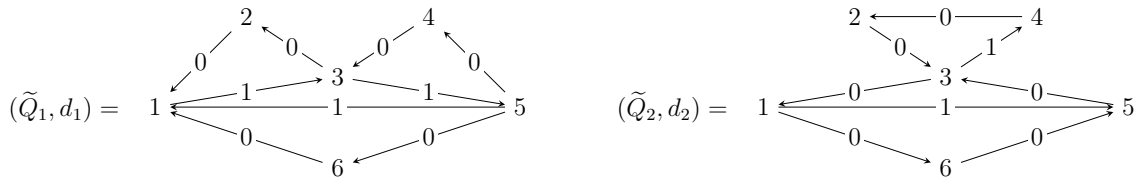
\[\scalebox{.8}{
\begin{tikzpicture}[>=stealth,scale=.5]
\node at (-3,0){$(\widetilde{Q}_1,d_1)=$};
\node (P1) at (0,0){$1$};
\node (P2) at (3,3){$2$};
\node (P3) at (6,1){$3$};
\node (P4) at (9,3){$4$};
\node (P5) at (12,0){$5$};
\node (P6) at (6,-2){$6$};
\node at (17,0){$(\widetilde{Q}_2,d_2)=$};
\node (Q1) at (20,0){$1$};
\node (Q2) at (23,3){$2$};
\node (Q3) at (26,1){$3$};
\node (Q4) at (29,3){$4$};
\node (Q5) at (32,0){$5$};
\node (Q6) at (26,-2){$6$};
\draw [->] (P2) -- node [fill=white,inner sep=.5mm]{0} (P1);
\draw [->] (P3) -- node [fill=white,inner sep=.5mm]{0}(P2);
\draw [->] (P4) -- node [fill=white,inner sep=.5mm]{0}(P3);
\draw [->] (P5) -- node [fill=white,inner sep=.5mm]{0}(P4);
\draw [->] (P5) -- node [fill=white,inner sep=.5mm]{0}(P6);
\draw [->] (P6) -- node [fill=white,inner sep=.5mm]{0}(P1);
\draw [<-] (P3) -- node [fill=white,inner sep=.5mm]{1}(P1);
\draw [<-] (P5) -- node [fill=white,inner sep=.5mm]{1}(P3);
\draw [->] (P5) -- node [fill=white,inner sep=.5mm]{1}(P1);

\draw [->] (Q3) -- node [fill=white,inner sep=.5mm]{0}(Q1);
\draw [->] (Q2) -- node [fill=white,inner sep=.5mm]{0}(Q3);
\draw [->] (Q4) -- node [fill=white,inner sep=.5mm]{0}(Q2);
\draw [->] (Q5) -- node [fill=white,inner sep=.5mm]{0}(Q3);
\draw [->] (Q6) -- node [fill=white,inner sep=.5mm]{0}(Q5);
\draw [->] (Q1) -- node [fill=white,inner sep=.5mm]{0}(Q6);
\draw [->] (Q3) -- node [fill=white,inner sep=.5mm]{1}(Q4);
\draw [->] (Q1) -- node [fill=white,inner sep=.5mm]{1}(Q5);
\end{tikzpicture}}
\]
Now applying the left graded mutations $\mu_6^L\circ \mu_3^L$ to the graded quiver with potential $(\widetilde{Q}_1,W_1,d_1)$ one can check that we obtain
\[\scalebox{.9}{
\begin{tikzpicture}[>=stealth,scale=.5]
\node at (15,0){$\mu_6^L\circ \mu_3^L(\widetilde{Q}_1,d_1)=$};
\node (Q1) at (20,0){$1$};
\node (Q2) at (23,3){$2$};
\node (Q3) at (26,1){$3$};
\node (Q4) at (29,3){$4$};
\node (Q5) at (32,0){$5$};
\node (Q6) at (26,-2){$6$};

\draw [->] (Q3) -- node [fill=white,inner sep=.5mm]{0}(Q1);
\draw [->] (Q2) -- node [fill=white,inner sep=.5mm]{0}(Q3);
\draw [->] (Q4) -- node [fill=white,inner sep=.5mm]{0}(Q2);
\draw [->] (Q5) -- node [fill=white,inner sep=.5mm]{-1}(Q3);
\draw [->] (Q6) -- node [fill=white,inner sep=.5mm]{1}(Q5);
\draw [->] (Q1) -- node [fill=white,inner sep=.5mm]{0}(Q6);
\draw [->] (Q3) -- node [fill=white,inner sep=.5mm]{1}(Q4);
\draw [->] (Q1) -- node [fill=white,inner sep=.5mm]{2}(Q5);
\end{tikzpicture}}
\]
By Theorem~\ref{Zgradedbirs}, we have an isomorphism of $\ZZ$-graded algebras $$\Jac(\mu_6^L\circ \mu_3^L(\widetilde{Q}_1,W_1,d_1))\underset{\ZZ}{\simeq} \End_{\Cc_1}(\pi_1(\mu_{P_6}^L\circ \mu_{P_3}^L(\Lambda_1)).$$
It is immediate to see that the $\ZZ$-graded algebras $\Jac(\mu_6^L\circ \mu_3^L(\widetilde{Q}_1,W_1,d_1))$ and $\Jac(\widetilde{Q}_2,W_2,d_2)$ are graded equivalent. Therefore, by Corollary~\ref{graded mutation and derived equivalence}, the algebras $\Lambda_1$ and $\Lambda_2$ are derived equivalent.

Now if we apply the left graded mutations $\mu_6^L\circ \mu_2^L\circ \mu_4^L$ to the graded quiver with potential $(\widetilde{Q}_1,W_1,d_1)$ one can check that we obtain the acyclic graded quiver

\[ (Q', d') := \mu_6^L\circ \mu_2^L\circ \mu_4^L(\widetilde{Q}_1,d_1) =
\begin{tikzpicture}[>=stealth,scale=.5,baseline=-1mm]
\node (P1) at (0,0){$1$};
\node (P2) at (3,3){$2$};
\node (P3) at (6,1){$3$};
\node (P4) at (9,3){$4$};
\node (P5) at (12,0){$5$};
\node (P6) at (6,-2){$6$};
\draw [->] (P1) -- node [fill=white,inner sep=.5mm] {0} (P2);
\draw [->] (P2) -- node [fill=white,inner sep=.5mm] {1} (P3);
\draw [->] (P3) -- node [fill=white,inner sep=.5mm] {0} (P4);
\draw [->] (P4) -- node [fill=white,inner sep=.5mm] {1} (P5);
\draw [->] (P6) -- node [fill=white,inner sep=.5mm] {1} (P5);
\draw [->] (P1) -- node [fill=white,inner sep=.5mm] {0} (P6);
\end{tikzpicture}
\]
It is easy to see that $(Q',d')$ is not graded equivalent to $(Q',0)$. 
Therefore Theorem~\ref{derivedeq1mutation} does not yield a derived equivalence between $\Lambda_1$ and $kQ'$. In fact, since there is an oriented cycle in the quiver of $\Lambda_2$, we know that $\Lambda_2$, and hence also $\Lambda_1$, is not piecewise hereditary.
\end{exa}

\section{Triangulated orbit categories} \label{section_triang_orbit}

This section is devoted to recalling some results of Keller \cite{Kel3,Kel,Kel2} (see also the appendix of \cite{IO2}), and to applying them to our setup.

\subsection{Pretriangulated DG categories}

\begin{dfa}
A \emph{DG category} is a $\ZZ$-graded category (i.e.\ morphism spaces are $\ZZ$-graded, and composition of morphisms respects this grading) with a differential $d$ of degree 1 satisfying the Leibniz rule.

\noindent 
For a DG category $\Xx$ we denote by $H^0\Xx$ the category with the same objects as $\Xx$ and with $$\Hom_{H^0\Xx}(X,Y):=H^0(\Hh om^\bullet_\Xx(X,Y)).$$
\end{dfa}

\begin{exa}
Let $\Aa$ be an additive category.
Then the class $\Cc(\Aa)_{dg}$ of complexes over $\Aa$ becomes a DG category if we set:
$$\Hh om^n_{\Aa}(X,Y):=\coprod_{i\in \ZZ}\Hom_\Aa(X^i,Y^{i+n}), \textrm{ and}$$
$$d((f_i)_{i\in\ZZ})=(f_id_Y-(-1)^nd_Xf_{i+1}), \ \textrm{for } (f_i)_{i\in\ZZ}\in \Hh om^n_\Aa(X,Y).$$
Then it is not hard to check that $Z^0(\Cc(\Aa)_{dg}) \simeq \Cc(\Aa)$ (where $Z^0$ is the kernel of the differential $d^0$), the category of complexes, and $H^0(\Cc(\Aa)_{dg})\simeq  \Hh(\Aa)$ the homotopy category of complexes over $\Aa$.
\end{exa}

The opposite category of a DG category and the tensor product of two DG categories are DG categories again (see \cite{Kel3} -- one has to be careful with the signs).

\begin{dfa}
Let $\Xx$ be a DG category. A DG $\Xx$-module is a DG functor $\Xx^{\rm op}\rightarrow \Cc(\Mod k)_{dg}$. The DG $\Xx$-modules form a DG category again, that we also denote by $\Cc(\Xx)_{dg}$ by abuse of notation.  Whenever we say two DG $\Xx$-modules are isomorphic, we mean they are isomorphic in $Z^0(\Cc(\Xx)_{dg})$.  We denote by $\Dd\Xx$ the category obtained from $H^0(\Cc(\Xx)_{dg})$ by inverting quasi-isomorphisms.

A DG $\Xx$-module is \emph{representable}, if it is isomorphic to a DG $\Xx$-module of the form $\Hh om^\bullet_\Xx(-,X)$ for some object $X$ in $\Xx$.

We denote by $\pretr \Xx$ the \emph{pretriangulated hull} of $\Xx$, i.e.\ the smallest subcategory of $\Cc(\Xx)_{dg}$ which contains the representable DG $\Xx$-modules, and which is closed under mapping cones (of morphisms in $Z^0(\Cc(\Xx)_{dg})$) and translations.
\end{dfa}

Note that by the Yoneda lemma, the natural DG functor $\Hh om^\bullet_\Xx(-,?) \colon \Xx\rightarrow \pretr\Xx$ is fully faithful.
 We call a DG category $\Xx$ \emph{pretriangulated} if the Yoneda functor is dense.

\begin{rema}
If  $F \colon \Xx\rightarrow \Yy$ is a DG functor between DG categories, it induces an induction functor $F^* \colon \Cc(\Xx)_{dg}\rightarrow \Cc(\Yy)_{dg}$. It sends representable functors to representable functors, and hence it induces a DG functor $F^* \colon \pretr \Xx \rightarrow \pretr \Yy$.
\end{rema}

\begin{prop}[Keller \cite{Kel2}]
Let $\Xx$ be a pretriangulated DG category. Then $H^0(\Xx)$ is an algebraic triangulated  category. Moreover any algebraic triangulated category comes up in this construction.
\end{prop}

\begin{exa}\label{exampleXY}
Let $\Lambda$ be an algebra of finite global dimension. 
\begin{itemize}
\item Let $\Xx:=\Cc^b(\proj \Lambda)_{dg}$ be the DG category of bounded complexes of finitely generated projective $\Lambda$-modules. Then $\Xx$ is pretriangulated and the triangulated category $H^0(\Xx)$ is equivalent to $\Dd^b(\Lambda)$.
\item Similarly, assume that $\Lambda$ is $G$-graded, where $G$ is an abelian group. Let $\Yy:=\Cc^b(\proj \cov{\Lambda}{G})_{dg}$ be the DG category of bounded complexes of finitely generated projective $\cov{\Lambda}{G}$-modules. Then $\Yy$ is pretriangulated and the triangulated category $H^0(\Yy)$ is equivalent to $\Dd^b(\cov{\Lambda}{G})\simeq \Dd^b(\gr \Lambda)$.
\end{itemize}
\end{exa}

\begin{dfa}
Let $\Xx$ and $\Yy$ be DG categories. We denote by $\rep(\Xx,\Yy)$ the full subcategory of  $\Dd(\Xx^{\rm op}\ten \Yy)$ formed by the objects $R$, such that for all $X\in\Xx$, the object $R(X\ten-)$ is isomorphic to a representable DG $\Yy$-module in $\Dd(\Yy)$.
\end{dfa}

\begin{exa}
Let $F \colon \Xx\rightarrow \Yy$ be a DG functor. Then $F$ induces an object $R_F$ in $\Dd(\Xx^{\rm op}\ten \Yy)$ given by $R_F(X\ten Y):= \Hh om^\bullet_\Yy (Y,FX)$. Since $R_F(X\ten -)$ is represented by $FX$, $R_F$ is in $\rep(\Xx,\Yy)$.
\end{exa}

\subsection{Universal property}

\begin{dfa}
Let $\Xx$ be a DG category, and $S \colon \Xx\rightarrow \Xx$ a DG functor inducing an equivalence on $H^0(\Xx)$. Then the \textit{DG orbit category }$\Xx/S$ has the same objects as $\Xx$, and 
$$\Hh om^\bullet_{\Xx/S}(X,Y):=\underset{p\gg 0}{\colim} \bigoplus_{i\geq 0}\Hh om_{\Xx}^\bullet(S^pX,S^{p+i}Y)\simeq \bigoplus_{i\in\ZZ}\underset{p\gg 0}{\colim}\Hh om^\bullet_{\Xx}(S^pX,S^{p+i}Y).$$
\end{dfa}

\begin{dfa}\label{deforbitcat}
Let $\Tt:=H^0(\Xx)$ be an algebraic triangulated category, and $S \colon \Xx\rightarrow \Xx$ be a DG functor inducing an equivalence  on $\Tt$. Then the \emph{triangulated orbit category} of $\Tt$ modulo $S$ is defined to be $$(\Tt/S)_\Delta:=H^0(\pretr(\Xx/S)).$$
There is a natural DG functor $\pi_\Xx \colon \Xx\rightarrow \Xx/S$ which induces a triangle functor $$\xymatrix{\pi_\Tt:=H^0(\pi_\Xx) \colon \Tt\ar[r] & (\Tt/S)_\Delta }.$$
\end{dfa}

\begin{rema}
The notation $(\Tt/S)_\Delta$ is not strictly justified. Indeed the triangulated category $H^0(\pretr(\Xx/S))$ depends on $\Xx$ and  $S \colon \Xx\rightarrow \Xx$ and not only on $\Tt$ and  $H^0S$. But in this paper the triangulated categories that we use have canonical enhancement in DG categories, and the auto-equivalences have canonical lifts.
\end{rema}

\begin{exa}\label{defclustercat} We can now give a more precise definition of the generalized cluster category given in Section~\ref{section_backgr}.

Let $\Lambda$ be a $\tau_2$-finite algebra of global dimension~$\leq 2$. Let $\Xx:=\Cc^b(\proj \Lambda)_{dg}$ be the DG category of bounded complexes of finitely generated projective $\Lambda$-modules, and let $S \colon \Xx\rightarrow \Xx$ be the DG functor $S:=-\ten_\Lambda p_{\Lambda^{\rm op}\ten \Lambda}D\Lambda [-2]$ where $  p_{\Lambda^{\rm op}\ten \Lambda}D\Lambda$ is a projective resolution of $D\Lambda$ as a $\Lambda$-$\Lambda$-bimodule. Then we have $$H^0(\Xx/S)\simeq \Dd^b(\Lambda)/\SSS \quad \textrm{and}\quad H^0(\pretr(\Xx/S))\simeq (\Dd^b(\Lambda)/\SSS)_\Delta=:\Cc_\Lambda.$$
\end{exa}

We are now ready to state a consequence of the universal property of the triangulated orbit category.
\begin{prop}[\cite{Kel}] \label{universalproperty}
Let $\Tt:=H^0(\Yy)$ and $\Tt':=H^0(\Xx)$ be two algebraic triangulated categories, and $S \colon \Yy \rightarrow \Yy$ be a DG functor inducing an equivalence on $\Tt$. Let $F \colon \Yy \rightarrow \Xx$ be a DG functor, and assume that there is an isomorphism in $\rep(\Yy,\Xx)$
$$F\circ S\simeq F.$$ Then $F$ induces a triangulated functor $(\Tt/S)_\Delta\rightarrow \Tt'$ such that the following diagram commutes
$$\xymatrix{\Tt\ar[r]^{H^0F}\ar[d]_{\pi_\Tt} & \Tt' \\ (\Tt/S)_\Delta\ar[ur] & }.$$
\end{prop}

\begin{cora}\label{hullofgraded}
Let  $\Lambda$ be a $\ZZ$-graded algebra of finite global dimension. Then we have a triangle equivalence:
 $$(\Dd^b(\cov{\Lambda}{\ZZ})/\langle 1 \rangle)_\Delta\simeq \Dd^b(\Lambda).$$
\end{cora}

\begin{proof}
Let $\Xx:=\Cc^b(\proj \Lambda)_{dg}$ and $\Yy:=\Cc^b(\proj \cov{\Lambda}{\ZZ})_{dg}$ be the DG categories defined in Example~\ref{exampleXY}.

The autoequivalence $\langle 1 \rangle$ of $\Dd^b(\cov{\Lambda}{\ZZ})$ is induced by tensoring with the $\Yy$-$\Yy$-bimodule $\leftsub{\rm id}{\cov{\Lambda}{\ZZ}}_{\left< 1 \right>}$ (that is, the right $\Yy$-action is twisted by $\langle 1 \rangle$).

The forgetful functor $\Dd^b(\cov{\Lambda}{\ZZ}) \rightarrow \Dd^b(\Lambda)$ is induced by the $\Yy$-$\Xx$-bimodule $\cov{\Lambda}{\ZZ}$ (obtained from the $\Yy$-$\Yy$-bimodule $\cov{\Lambda}{\ZZ}$ by applying the forgetful functor on the right side).

It is clear that we have an isomorphism of $\cov{\Lambda}{\ZZ}$-$\Lambda$-bimodules
$$\leftsub{\rm id}{\cov{\Lambda}{\ZZ}}_{\langle 1 \rangle} \simeq \cov{\Lambda}{\ZZ}$$
Therefore, by Proposition~\ref{universalproperty}, there exists a triangle functor $G$
$$\xymatrix{\Dd^b(\cov{\Lambda}{\ZZ})\ar[rr] \ar[dr] & & \Dd^b(\Lambda) \\ & (\Dd^b(\cov{\Lambda}{\ZZ})/\langle 1 \rangle)_{\Delta}\ar@{-->}[ur]^G &}$$
The functor $\xymatrix{\Dd^b(\cov{\Lambda}{\ZZ})/\langle 1 \rangle\ar[r] & \Dd^b(\Lambda)}$ is fully faithful, hence so is $G$. Moreover $\Dd^b(\Lambda)$ is generated as triangulated category by the simples, which are in the image of $G$. Therefore the functor $G$ is an equivalence of triangulated categories.
\end{proof}

\begin{prop}\label{universalprop2}
Let $\Tt:=H^0(\Xx)$ and $\Tt':=H^0(\Yy)$ be two algebraic triangulated categories, and $S \colon \Xx\rightarrow \Xx$ (resp.\ $S' \colon \Yy\rightarrow\Yy$) be a DG functor inducing an equivalence on $\Tt$ (resp.\ on $\Tt'$). Let $F \colon \Xx\rightarrow \Yy$ be a DG functor, and assume that there is an isomorphism in $\rep(\Xx,\Yy)$
$$F\circ S\simeq S'\circ F.$$ Then $F$ induces a triangulated functor $(\Tt/S)_\Delta \rightarrow (\Tt'/S')_\Delta$ such that the following diagram commutes.
$$\xymatrix{\Tt \ar[r]^{H^0(F)}\ar[d]_{\pi_\Tt} & \Tt'\ar[d]^{\pi_{\Tt'}} \\ (\Tt/S)_\Delta \ar[r] & (\Tt'/S')_\Delta}$$
\end{prop}

\begin{proof}
Let $\pi_\Yy$ be the DG functor $\Yy\rightarrow \pretr(\Yy/S')$. Then we have $\pi_\Yy\circ S'\simeq \pi_\Yy$ in $\rep(\Yy,\pretr(\Yy/S'))$ by definition. Therefore we have isomorphism in $\rep(\Xx,\pretr(\Yy/T))$: $$(\pi_\Yy\circ F)\circ S\simeq \pi_\Yy \circ S'\circ F\simeq \pi_\Yy\circ F.$$
Hence by Proposition~\ref{universalproperty}, we get a commutative diagram:
$$\xymatrix{\Tt=H^0(\Xx)\ar[r]^{H^0(F)}\ar[d]_{H^0(\pi_\Xx)} & \Tt'=H^0(\Yy)\ar[d]^{H^0(\pi_\Yy)} \\ H^0(\pretr(\Xx/S))\ar[r]^f & H^0(\pretr(\Yy/S'))}.$$
 where $f$ is a triangle functor.
\end{proof}

\begin{cora}\label{derivedeqclustereq}
Let $\Lambda_1$ and $\Lambda_2$ be derived equivalent algebras of global dimension~$\leq 2 $ which are $\tau_2$-finite. Then they are cluster equivalent.
\end{cora}

\begin{proof}
Since $\Lambda_1$ and $\Lambda_2$ are derived equivalent, there exists $T\in \Dd^b(\Lambda_1^{\rm op}\ten\Lambda_2)$ such that $$\xymatrix{\Dd^b(\Lambda_1)\ar[rr]^{-\lten_{\Lambda_1}T} && \Dd^b(\Lambda_2)}$$ is an equivalence.
By the previous proposition it is enough to check that there exists an isomorphism in $\Dd^b(\Lambda_1^{\rm op}\ten\Lambda_2)$ between $D\Lambda_1\lten_{\Lambda_1}T$ and $T\lten_{\Lambda_2}D\Lambda_2$.

Using the isomorphisms in $\Dd^b(\Lambda_1^{\rm op}\ten\Lambda_2)$
\begin{align*}
D\Lambda_1\lten_{\Lambda_1}T & \simeq D\Hom_{\Lambda_2}(T,T)\lten_{\Lambda_1}T \\
& \simeq \Hom_{\Lambda_2}(T,T\lten_{\Lambda_2}D\Lambda_2) \lten_{\Lambda_1}T,
\end{align*}
we get a natural morphism $D\Lambda_1\lten_{\Lambda_1}T\rightarrow  T\lten_{\Lambda_2}D\Lambda_2$ in $\Dd^b(\Lambda_1^{\rm op}\ten\Lambda_2)$ induced by the evaluation morphism. This morphism is clearly an isomorphism in $\Dd^b(\Lambda_2)$ by uniqueness of the Serre functor. Hence it is an isomorphism in $\Dd^b(\Lambda_1^{\rm op}\ten\Lambda_2)$.
\end{proof}

\subsection{Iterated triangulated orbit categories}
\begin{prop}\label{commutativeorbit}
Let $\Tt:=H^0(\Xx)$ be an algebraic triangulated category, and $S,T \colon \Xx\rightarrow \Xx$ be DG functors inducing equivalences on $\Tt$,  such that there is a natural isomorphism $S\circ T\simeq T\circ S$. Then there is an equivalence of DG categories $$\pretr(\pretr(\Xx/S)/T) \simeq \pretr(\pretr(\Xx/T)/S).$$
Therefore, there is a triangle equivalence: 
$$( (\Tt/S)_\Delta)/T)_\Delta \simeq ( (\Tt/T)_\Delta)/S)_\Delta.$$
\end{prop}

\begin{proof}
We divide the proof into two lemmas.

\begin{lema}
Under the hypothesis of Proposition~\ref{commutativeorbit} there is
 an equivalence of DG categories $$(\Xx/S)/T\simeq (\Xx/T)/S.$$ 
\end{lema}
\begin{proof}
The functor $\colim$ is a left adjoint, hence it commutes with colimits. Therefore, since $S$ and $T$ commute, we have
\begin{align*}
\Hh om^\bullet_{(\Xx/S)/T}(X,Y) & = \underset{q}{\colim}\bigoplus_{j\geq 0}\underset{p}{\colim} \bigoplus_{i\geq 0}\Hh om^\bullet_{\Xx}(T^{q}S^pX,T^{q+j}S^{p+i}Y)\\
& \simeq  \underset{p}{\colim}\bigoplus_{i\geq 0}\underset{q}{\colim} \bigoplus_{j\geq 0}\Hh om^\bullet_{\Xx}(T^{q}S^pX,T^{q+j}S^{p+i}Y)\\
& \simeq \underset{p}{\colim}\bigoplus_{i\geq 0}\underset{q}{\colim} \bigoplus_{j\geq 0}\Hh om^\bullet_{\Xx}(S^pT^{q}X,S^{p+i}T^{q+j}Y)\\
& \simeq \Hh om^\bullet_{(\Xx/T)/S}(X,Y) \qedhere
\end{align*}
\end{proof}
We denote by $\Xx/\{S,T\}$ the DG category $(\Xx/S)/T\simeq (\Xx/T)/S$. 

\begin{lema} \label{lemma.pretr_twofunct}
Under the hypothesis of Proposition~\ref{commutativeorbit} there is an equivalence of DG categories
 $$\xymatrix{ \pretr(\Xx/\{S,T\})\ar[r]^(.45)\sim & \pretr(\pretr(\Xx/S)/T)}.$$
\end{lema}
\begin{proof}
We have a fully faithful DG functor $$\xymatrix{\Xx/S\ar@{^(->}[r] &\pretr(\Xx/S)}.$$ It induces a fully faithful DG functor $$\xymatrix{\Xx/\{S,T\}\ar@{^(->}[r] & \pretr(\Xx/S)/T\ar@{^(->}[r] & \pretr(\pretr(\Xx/S)/T)}.$$
Since the category $\pretr(\pretr(\Xx/S)/T)$ is a pretriangulated DG category, we get a fully faithful DG functor $$\xymatrix{\pretr(\Xx/\{S,T\})\ar@{^(->}[r] & \pretr(\pretr(\Xx/S)/T)}.$$ Now the objects of $\pretr(\pretr(\Xx/S)/T)$ are iterated cones of objects in $\Xx$, hence this DG functor is also dense.
Thus we get an equivalence of DG categories
\[ \xymatrix{ \pretr(\Xx/\{S,T\})\ar[r]^(.45)\sim & \pretr(\pretr(\Xx/S)/T)}. \qedhere \]
\end{proof}
We can now prove Proposition~\ref{commutativeorbit}.
Using Lemma~\ref{lemma.pretr_twofunct}  and its symmetric version (interchanging $S$ and $T$), we immediately get a DG equivalence 
$$\pretr(\pretr(\Xx/S)/T) \simeq \pretr(\pretr(\Xx/T)/S).$$ By taking the $H^0$ of each side we get the triangle equivalence
\[ ( (\Tt/S)_\Delta)/T)_\Delta \simeq ( (\Tt/T)_\Delta)/S)_\Delta. \qedhere \]
\end{proof}

\begin{cora}\label{doubleorbit}
Let $\Lambda$ be a $\ZZ$-graded algebra of global dimension~$\leq 2$ which is $\tau_2$-finite. Let $\cov{\Lambda}{\ZZ}$ be the $\ZZ$-covering of $\Lambda$, and $\SSS:=- \lten_{\cov{\Lambda}{\ZZ}} D \cov{\Lambda}{\ZZ} [-2]$. Then there is a triangle equivalence
$$\left((\Dd^b(\cov{\Lambda}{\ZZ})/\SSS)_{\Delta}/\langle 1 \rangle\right)_{\Delta}\simeq \Cc_{\Lambda}.$$
\end{cora}

\begin{proof}
We apply Proposition~\ref{commutativeorbit} for
$\Xx:= \Cc^b(\proj \cov{\Lambda}{\ZZ})_{dg}$, $S:=-\otimes_{\cov{\Lambda}{\ZZ}} P$, and $T:=\langle 1 \rangle$, where $P$ is a bimodule projective resolution of $D \cov{\Lambda}{\ZZ} [-2]$.

Then we clearly have $S\circ T\simeq T\circ S.$
Thus we get
\[ \left((\Dd^b(\cov{\Lambda}{\ZZ})/\SSS)_{\Delta}/\langle 1 \rangle\right)_{\Delta}\simeq H^0(\pretr(\pretr(\Xx/S)/T)) \simeq H^0(\pretr(\pretr(\Xx/T)/S))\simeq \Cc_\Lambda. \qedhere \]  
\end{proof}

\section{Graded derived equivalence for cluster equivalent algebras} \label{section_gr_der_eq}

\subsection{Graded version of results of Section~\ref{section_left_mutation}}

In this section we generalize the previous results to the case of a graded algebra $\Lambda$.

Let $\Lambda$ be a $\mathbb{Z}$-graded algebra of global dimension~$\leq 2$ which is $\tau_2$-finite. We denote by $\cov{\Lambda}{\ZZ}$ its $\ZZ$-covering. The functor $-\lten_{\cov{\Lambda}{\ZZ}}D \cov{\Lambda}{\ZZ} [-2]$ is an autoequivalence of $\Dd^b(\cov{\Lambda}{\ZZ})$ that we will denote by $\SSS$ by  abuse of notation. We denote by $\pi_{\Lambda}^{\ZZ}$ the composition
$$\xymatrix{\pi^\ZZ_\Lambda \colon \Dd^b(\cov{\Lambda}{\ZZ})\ar[r] & \Dd^b(\Lambda)\ar[r]^{\pi_\Lambda} & \Cc_\Lambda}.$$
This graded version of Proposition~\ref{preimage} is not hard to check:
\begin{prop}\label{graded2clustertilting} 
Let $\Lambda$ be a $\ZZ$-graded algebra which is $\tau_2$-finite and of global dimension~$\leq 2$. Let $T$ be a cluster-tilting object in $\Cc_\Lambda$. The subcategory $(\pi_{\Lambda}^\ZZ)^{-1}(T)$ is cluster-tilting subcategory of $\Dd^b(\cov{\Lambda}{\ZZ})$. 

In particular $\add\{\SSS^p\Lambda\langle q\rangle \mid p,q\in \ZZ\}=(\pi^\ZZ_\Lambda)^{-1}(\pi_\Lambda(\Lambda))$ is a cluster-tilting subcategory of $\Dd^b(\cov{\Lambda}{\ZZ})$.
\end{prop}

Here is the graded version of Proposition~\ref{Zgrading}.

\begin{prop}\label{Z2grading}
Let $Q$ be a $\ZZ$-graded quiver, and $I\subset kQ$ be an admissible ideal which is generated by homogeneous elements and such that the algebra $\Lambda=kQ/I$ is of global dimension~$\leq 2$ and $\tau_2$-finite. Denote by $(\widetilde{Q}, W)$ the quiver with potential defined in Theorem~\ref{keller}. 
Then there exists a unique $\ZZ^2$-grading on $\widetilde{Q}$ such that
\begin{enumerate}
\item the potential $W$ is homogeneous of degree $(1,1)$;
\item there is an isomorphism of $\ZZ$-graded quivers $\xymatrix{\widetilde{Q}^{\{0\}\times\ZZ}\ar[r]^(.6)\sim_(.6){\ZZ} & Q}$.
\end{enumerate}
This grading on $\widetilde{Q}$ yields a grading on $\Jac(\widetilde{Q},W)$ and we have an isomorphism of $\ZZ^2$-graded algebras
$$\xymatrix{\Jac(\widetilde{Q},W)\ar[r]^(.3)\sim_(.3){\ZZ^2} &\bigoplus_{p,q\in\ZZ}\Hom_{\Dd^\ZZ}(\Lambda\langle 0\rangle,\SSS^{-p}\Lambda\langle -p+q\rangle)}.$$
\end{prop}

\begin{proof}
There are two kinds of arrows in the quiver $\widetilde{Q}$: arrows of $Q$ and arrows coming from minimal relations. By $(2)$ any arrow $a$ coming from an arrow of $Q$ has to be of degree $(0,{\rm deg}(a))$. Let $r$ be a minimal relation in $I$. By definition ${\rm deg}(r)$ is well defined. Then by Condition~(1) since $ra_r$ is a term of $W$, the degree of $a_r$ has to be $(1,1-{\rm deg}(r))$. Hence we have existence and uniqueness of such a grading.

We have the following isomorphisms:
$$\widetilde{\Lambda}:=\End_{\Cc_\Lambda}(\pi(\Lambda))\simeq \End_{\Cc_\Lambda}(\pi^\ZZ(\Lambda\langle 0 \rangle))\simeq \bigoplus_{p,q\in\ZZ}\Hom_{\Dd^\ZZ}(\Lambda\langle 0\rangle,\SSS^{-p}\Lambda\langle -p+q\rangle).$$
Since $\Jac(\widetilde{Q},W)\simeq\widetilde{\Lambda}$ by Theorem~\ref{keller}, we just have to check that it respects the gradings previously defined. Let $a$ be an arrow of the quiver $Q$. It can be seen as an element of $\Hom_{\Dd^\ZZ}(\Lambda\langle 0 \rangle,\Lambda \langle {\rm deg}(a)\rangle)$ so as an element of degree $(0,{\rm deg}(a))$ of the algebra $\widetilde{\Lambda}=\bigoplus_{p,q\in\ZZ}\Hom_{\Dd^\ZZ}(\Lambda\langle 0\rangle,\SSS^{-p}\Lambda\langle -p+q\rangle).$ Now let $r$ be a minimal relation in $I$, and $a_r$ be the corresponding arrow in $\widetilde{Q}$. The minimal relation $r$ corresponds to an element of $\Ext^2_\Lambda(S_i,S_j\langle -{\rm deg}(r)\rangle)$ where $s(r)=j$ and $t(r)=i$. Hence it is an element in 
$$\Hom_{\Dd^\ZZ}(\Lambda\langle 0 \rangle, \SSS^{-1}\Lambda\langle -{\rm deg}(r)\rangle) = \Hom_{\Dd^\ZZ}(\Lambda\langle 0 \rangle, \SSS^{-1}\Lambda\langle -1+(-{\rm deg}(r)+1)\rangle)$$
so an element of degree $(1,-{\rm deg}(r)+1)$ in $\widetilde{\Lambda}$. Hence the isomorphism $\Jac(\widetilde{Q},W,d)\simeq\widetilde{\Lambda}$ given by Theorem~\ref{keller} is an isomorphism of $\ZZ^2$-graded algebras.
\end{proof}

Let $\Lambda$ be a $\ZZ$-graded algebra of global dimension
$\leq 2$ which is $\tau_2$-finite. Let $T$ be an object in $\Dd^b(\cov{\Lambda}{\ZZ})$ such that $\pi^\ZZ(T)$ is a
(basic) cluster-tilting object in $\Cc_\Lambda$. The endomorphism algebra $$\End_\Cc(\pi^\ZZ(T))=\bigoplus_{p,q\in
  \mathbb{Z}}\Hom_{\Dd^\ZZ}(T,\SSS^{-p}T\langle -p+q\rangle)$$ is naturally $\ZZ^2$-graded. 

Let $T_i$ be an indecomposable summand of $T\simeq T_i\oplus T'$.
We denote by $\Uu_i$ the additive subcategory $\add\{\SSS^p
T'\langle q\rangle \mid p,q\in \mathbb{Z}\}$ of $\Dd^b(\cov{\Lambda}{\ZZ})$. Consider a triangle in
$\Dd^b(\cov{\Lambda}{\ZZ})$
$$\xymatrix{T_i\ar[r]^u & B\ar[r]^v & T_i^L\ar[r]^w & T_i[1]}$$
where $u \colon T_i\rightarrow B$ is a
minimal left $\Uu_i$-approximation of $T_i$. We call this triangle the \emph{graded left exchange triangle} associated with $T_i$. We write $\mu_i^L(T):=T'\oplus T_i^L$. 
Since the image of the graded left exchange triangle in $\Dd^b(\cov{\Lambda}{\ZZ})$ is an exchange triangle in $\Cc_\Lambda$, the object $\pi^\ZZ(T'\oplus T^L_i)$ is cluster-tilting in $\Cc_\Lambda$.

It is also possible to consider the \emph{graded right exchange triangle} associated with $T_i$:
$$\xymatrix{T_i^R\ar[r]^{u'} & B'\ar[r]^{v'} & T_i\ar[r] & T_i^R[1]}.$$

\begin{thma}\label{Z2gradedbirs}
Let $\Lambda=kQ/I$ be a graded $\tau_2$-finite algebra of global dimension~$\leq 2$, and denote by $(\widetilde{Q},W,d)$ the $\ZZ^2$-graded QP defined in Proposition~\ref{Z2grading}. Assume that there exists a sequence $i_1,i_2,\ldots,i_l$ of vertices of $\widetilde{Q}$ such that for any $j=0,\ldots,l$ there is no $2$-cycle on the vertex $i_{j+1}$ in the quiver $Q^{j}$ where $(Q^j,W^j):=\mu_{i_{j}}\circ\cdots\circ \mu_{i_1}(\widetilde{Q},W)$. Denote by $T$ the object in $\Dd^\ZZ:=\Dd^b(\cov{\Lambda}{\ZZ})$ defined by $T:=\mu^L_{i_l}\circ\cdots\circ\mu_{i_1}^L(\Lambda\langle 0\rangle)$. Then there is an isomorphism of $\ZZ^2$-graded algebras
 $$\xymatrix{\bigoplus_{p\in\mathbb{Z}}\Hom_{\Dd^\ZZ}(T,\SSS^{-p}T\langle -p+q\rangle)\ar[r]^-\sim_-{\ZZ^2}& \Jac(\mu_{i_l}^L\circ\cdots\circ\mu_{i_1}^L(\widetilde{Q},W,d)).}$$  
\end{thma}

\begin{proof}
The proof is very similar to the proof of Theorem~\ref{Zgradedbirs}. We briefly outline the last step, which is a bit more technical.

Let $b \colon r\rightarrow i_{j+1}$ be an arrow in $Q^j$ and denote by $(x,y)$ the degree of $b$. We consider the graded right exchange triangle in $\Dd^b(\cov{\Lambda}{\ZZ})$
\[ \xymatrix{(T^j_{i_{j+1}})^R\ar[r]^-{u'} & B'\ar[r]^-{v'} & T^j_{i_{j+1}}\ar[r] & (T^j_{i_{j+1}})^R[1]}, \] where $T^j_{i_{j+1}}$ denotes the summand of $T^j:=\mu_{i_j}^L\circ\ldots\circ\mu_{i_1}^L(\Lambda\langle 0\rangle)\in\Dd^b(\cov{\Lambda}{\ZZ})$ corresponding to the vertex $i_{j+1}$.
Then $\SSS^{x}T^j_r\langle x-y\rangle$ is a direct summand of $B'$,
and the reverse arrow $b^*$ corresponds to the component $(T^j_{i_{j+1}})^R \to \SSS^{x}T^j_r\langle x-y\rangle$. Since, by Proposition~\ref{tirtil}, we have $(T^j_{i_{j+1}})^R \simeq \SSS (T^j_{i_{j+1}})^L$ we obtain that $b^*$
corresponds to a map
$$\xymatrix{(T^j_{i_{j+1}})^L\ar[r] & \SSS^{x-1}T^j_r\langle x-1-(y - 1)\rangle},$$
thus to a map of degree $(1-x,1-y)$. 
\end{proof}

\subsection{Graded derived equivalence}

In this section we generalize Theorem~\ref{derivedeq1} to the setup where the algebras are not graded equivalent. In this setup, we will not get a derived equivalence between the algebras $\Lambda_1$ and $\Lambda_2$, but a derived equivalence between their coverings $\cov{\Lambda_1}{\ZZ}$ and $\cov{\Lambda_2}{\ZZ}$, for suitable gradings on them.

In order to do that, we will use a graded version of the recognition theorem (Theorem~\ref{clustertilt}). The proof of this theorem is very similar to the proof of Theorem~\ref{clustertilt}.

\begin{thma}\label{Zclustertilt}
Let $\Tt$ be an algebraic triangulated category with a Serre functor and with a cluster-tilting subcategory $\Vv$. Let $\Lambda$ be a $\tau_2$-finite algebra with global dimension~$\leq 2$, and with a $\ZZ$-grading. Denote by $\Uu$ the cluster-tilting subcategory $\add\{\SSS^{-p}\Lambda\langle -p+q\rangle \mid p,q\in \mathbb{Z}\}$ of $\Dd^b(\cov{\Lambda}{\ZZ})$. Assume that there is an equivalence of additive categories with $\SSS$-action  $\xymatrix{f \colon \Uu \ar[r]^\sim & \Vv}$. Then there exists a triangle equivalence $F \colon \Dd^b(\cov{\Lambda}{\ZZ})\rightarrow \Tt$ such that the following diagram commutes
$$\xymatrix{\Dd^b(\cov{\Lambda}{\ZZ})\ar[r]^-F & \Tt\\ \Uu\ar[r]^f \ar@{^(->}[u] & \Vv\ar@{^(->}[u]}$$
\end{thma}

For the statement of the main theorem of this section, we will need this technical definition.

\begin{dfa}
Let $\Lambda_1$ and $\Lambda_2$ be two algebras of global dimension~$\leq 2$ which are $\tau_2$-finite. For $j=1,2$, we denote by $\pi_j$ the canonical functor $\pi_j \colon \Dd_j\rightarrow \Cc_j$ where $\Dd_j:=\Dd^b(\Lambda_j)$ and $\Cc_j=\Cc_{\Lambda_j}$. We will say that $\Lambda_1$ and $\Lambda_2$ satisfy the \emph{compatibility condition} if there exists a sequence $i_1,i_2,\ldots, i_l$ (called \emph{compatible sequence}) and a $\ZZ$-graded QP $(\widetilde{Q}^1,W^1,d^1)$ with $W^1$ homogeneous of degree $1$ such that:
\begin{enumerate}
\item we have an isomorphism of $\ZZ$-graded algebras $$\Jac(\widetilde{Q}^1,W^1,d^1)\underset{\ZZ}{\simeq} \bigoplus_{p\in\ZZ}\Hom_{\Dd_1}(\Lambda_1,\SSS^{-p}\Lambda_1)$$
\item for any $0\leq j\leq l$, the quiver of $\End_{\Cc_1}(T_{j})$ has neither loops nor 2-cycles at the vertex $i_{j+1}$, where $T_j:=\mu_{i_j}\circ\cdots\circ \mu_{i_2}\circ \mu_{i_1}(\pi_1\Lambda_1)$;
\item there exists an isomorphism $\End_{\Cc_1}(T_l)\simeq \End_{\Cc_2}(\pi_2\Lambda_2)$;
\item the isomorphisms in (1) and (3) can be chosen in such a way that there exists a $\ZZ$-grading $d^2$ on $(\widetilde{Q}^2,W^2):= \mu_{i_l}\circ\cdots\circ \mu_{i_2}\circ \mu_{i_1}(\widetilde{Q}^1,W^1)$ such that we have an isomorphism of $\ZZ$-graded algebras $$\Jac(\widetilde{Q}^2,W^2,d^2)\underset{\ZZ}{\simeq} \bigoplus_{q\in \ZZ}\Hom_{\Dd_2}(\Lambda_2,\SSS^{-q}\Lambda_2).$$ 
\end{enumerate}
\end{dfa}

\begin{rema}
\begin{enumerate}
\item The grading $d^2$ will typically not be the grading obtained by mutation on graded quivers with potential.
\item In this definition, (1), (2), (3) mean that the quivers with potential of $\widetilde{\Lambda}_1$ and $\widetilde{\Lambda}_2$ (see Theorem~\ref{keller}) are linked by a sequence of mutations, and that neither loops nor 2-cycles occur at any intermediate step of this sequence of mutations. (The condition of not having loops or 2-cycles is automatic if the QP is rigid.)
\item If $\widetilde{\Lambda}_1 \simeq \widetilde{\Lambda}_2$, then conditions (1), (2) and (3) hold. We then have two (possibly) different $\ZZ$-gradings on $\widetilde{\Lambda}_1=\widetilde{\Lambda}_2$. Condition (4) means that these two gradings yield a $\ZZ^2$-grading on $\widetilde{\Lambda}_1=\widetilde{\Lambda}_2$.
\item When conditions (1), (2) and (3) are satisfied, we do not know of any counterexamples to condition (4) being satisfied.
\end{enumerate}\end{rema}

\begin{thma}\label{maintheorem}
Let $\Lambda_1$ and $\Lambda_2$ be two algebras of global dimension~$\leq 2$ which are $\tau_2$-finite and which satisfy the compatibility condition.
Then 
\begin{enumerate}
\item there are $\ZZ$-gradings on $\Lambda_1$ and on $\Lambda_2$, such that there exists a derived equivalence $$\xymatrix{F^\ZZ \colon \Dd_1^\ZZ\ar[r]^(.6)\sim & \Dd_2^\ZZ},$$ where $\Dd_j^\ZZ:=\Dd^b(\cov{\Lambda_j}{\ZZ})$. 
\item the equivalence $F^\ZZ$ induces a triangle equivalence $\xymatrix{F \colon \Cc_1\ar[r]^(.55)\sim & \Cc_2}$ such that the following diagram commutes: $$\xymatrix@C=1.5cm{\Dd_1^\ZZ\ar[d] \ar@/_7mm/[dd]_{\pi_1^\ZZ}\ar[r]^\sim_{F^\ZZ} & \Dd_2^\ZZ\ar[d] \ar@/^7mm/[dd]^{\pi_2^\ZZ} \\ \Dd_1\ar[d]^{\pi_1} & \Dd_2\ar[d]_{\pi_2} \\ \Cc_1 \ar[r]^\sim_F & \Cc_2}$$
\end{enumerate}

\end{thma} 

\begin{proof}[Proof of (1)]

Take $(\widetilde{Q}^1,W^1,d^1)$  and the isomorphism (1) of the compatibility condition 
$$\xymatrix{\Jac(\widetilde{Q}^1,W^1, d^1)\ar[r]^(.4)\sim_(.4){\ZZ} &\bigoplus_{p\in\ZZ}\Hom_{\Dd_1}(\Lambda_1,\SSS^{-p}\Lambda_1)}.$$
Let $s:=i_1,\ldots, i_l$ be a compatible sequence and define $T^1$ as the object
\[ T^1:=\mu^L_{i_l}\circ\cdots\circ \mu^L_{i_2}\circ \mu^L_{i_1}(\Lambda_1)\in \Dd^b(\Lambda_1). \]
By Condition~(2) of the compatible sequence we can apply Theorem~\ref{Zgradedbirs} and we get an isomorphism:
$$\xymatrix{\Jac(\mu_s^L(\widetilde{Q}^1,W^1,d^1))\ar[r]^(.45)\sim_(.45){\ZZ} &\bigoplus_{p\in\ZZ}\Hom_{\Dd_1}(T^1,\SSS^{-p}T^1)},$$
where $\mu_s^L$ is the composition $\mu_{i_l}^L\ldots\mu_{i_1}^L$.

By Condition~(4) of the compatibility condition,  there exists a $\ZZ$-grading $d^2$ on $(\widetilde{Q}^2,W^2):=\mu_{i_l}\ldots \mu_{i_1}(\widetilde{Q}^1,W^1)$ such that $W^2$ is homogeneous of degree $1$ and such that we have an isomorphism
$$\xymatrix{\Jac(\widetilde{Q}^2,W^2,d^2)\ar[r]^(.4)\sim_(.4){\ZZ} &\bigoplus_{q\in\ZZ}\Hom_{\Dd_2}(\Lambda_2,\SSS^{-p}\Lambda_2)}.$$
Therefore the $\ZZ^2$-grading $(\mu_s^L(d^1),d^2)$ on $\widetilde{Q}^2$ makes $W^2$ homogeneous of degree $(1,1)$. Moreover we have an isomorphism $$\Jac(\widetilde{Q}^2,W^2,(\mu_s^L(d^1),d^2))^{\ZZ\times\{0\}}\simeq \Lambda_2,$$ hence we get a $\ZZ$-grading on $\Lambda_2$. By the uniqueness of the $\ZZ^2$-grading of Proposition~\ref{Z2grading} we have an isomorphism of $\ZZ^2$-graded algebras:
$$\xymatrix{\Jac(\widetilde{Q}_2,W^2,(\mu_s^L(d^1),d^2))\ar[r]^(.4)\sim_(.4){\ZZ^2} &\bigoplus_{p,q\in\ZZ}\Hom_{\Dd_2^\ZZ}(\Lambda_2\langle 0\rangle,\SSS^{-q}\Lambda_2\langle -q+p\rangle)}.$$

We define $T^2\in \Dd^b(\cov{\Lambda_2}{\ZZ})$ as $T^2:=\mu_{i_{1}}^R\cdots \mu_{i_l}^R(\Lambda_2\langle 0 \rangle)$. By Condition~(2) of the compatible sequence we can apply Theorem~\ref{Z2gradedbirs} and we get an isomorphism:
$$\xymatrix{\Jac(\mu_s^R(\widetilde{Q}^2,W^2,\mu^L_s(d^1),d^2))\ar[r]^(.45)\sim_(.45){\ZZ^2} &\bigoplus_{p,q\in\ZZ}\Hom_{\Dd^\ZZ_2}(T^2,\SSS^{-q}T^2\langle -q+p\rangle)},$$
where $\mu_s^R$ is the composition $\mu_{i_1}^R\ldots\mu_{i_l}^R$.
By Lemma~\ref{lemma left=right inverse} the graded QP $\mu_s^R(\widetilde{Q}^2,W^2,\mu^L_s(d^1))=\mu^R_s(\mu_s^L(\widetilde{Q}^1,W^1,d^1))$ is graded right equivalent to $(\widetilde{Q}^1,W^1,d^1)$.
 Therefore we have an isomorphism 
\[ \tag{$*$} \xymatrix{\Jac(\widetilde{Q}^1,W^1,( d^1,\mu_s^R(d^2)))\ar[r]^(.42)\sim_(.42){\ZZ^2} &\bigoplus_{p,q\in\ZZ}\Hom_{\Dd^\ZZ_2}(T^2,\SSS^{-q}T^2\langle -q+p\rangle)} \]
By definition of $d^1$ there is an isomorphism $$\Lambda_1\simeq \Jac((\widetilde{Q}^1,W^1,d^1,\mu_s^R(d^2)))^{\{0\}\times \ZZ}.$$ Therefore we get a $\ZZ$-grading on $Q^1$. The ideal $I^1$ is generated by $\{\partial_aW^1, d^1(a)=1\}$, hence is generated by elements which are homogeneous with respect to the grading $\mu_s^R(d^2)$. By the uniqueness of the $\ZZ^2$-grading of Proposition~\ref{Z2grading} we have 
\[ \tag{$\dagger$} \xymatrix{\Jac(\widetilde{Q}_1,W^1,( d^1,\mu^R_s(d^2)))\ar[r]^(.4)\sim_(.4){\ZZ^2} &\bigoplus_{p,q\in\ZZ}\Hom_{\Dd_1^\ZZ}(\Lambda_1\langle 0\rangle,\SSS^{-p}\Lambda_1\langle -p+q\rangle)}. \]
Finally, combining $(*)$ and $(\dagger)$ we get
\[ \tag{$\ddagger$} \xymatrix{ \bigoplus_{p,q\in\ZZ}\Hom_{\Dd^\ZZ_2}(T^2,\SSS^{-q}T^2\langle -q+p\rangle) \ar[r]^(.5)\sim_(.5){\ZZ^2} &\bigoplus_{p,q\in\ZZ}\Hom_{\Dd_1^\ZZ}(\Lambda_1\langle 0\rangle,\SSS^{-p}\Lambda_1\langle -p+q\rangle)}. \]

The category $\add\{\SSS^pT^2\langle q\rangle \mid p,q\in \ZZ\}=(\pi^\ZZ_2)^{-1}(\mu_S \pi_2 \Lambda_2)$ is cluster-tilting in $\Dd^b(\cov{\Lambda_2}{\ZZ})$ by Proposition~\ref{graded2clustertilting}. The isomorphism $(\ddagger)$ implies that we have an equivalence of $\ZZ^2$-categories 
$$\xymatrix{\add\{\SSS^{-p}\Lambda_1\langle -p+q\rangle \mid p,q\in \ZZ\}\ar[r]^\sim_f & \add\{\SSS^{-q}T^2\langle -q+p\rangle \mid p,q\in \ZZ\}}.$$ In order to apply Theorem~\ref{Zclustertilt}, we have to check that $f$ commutes with $\SSS$:
\begin{align*}
f(\SSS(\SSS^{-p}\Lambda_1\langle -p+q\rangle)) &= f (\SSS^{-p+1}\Lambda_1\langle -p+1 +(q-1)\rangle)\\
& = \SSS^{-q+1}T^2\langle -q+1 + (p-1)\rangle \\
& = \SSS(\SSS^{-p}T^2\langle -q+p\rangle) \\
& = \SSS f(\SSS^{-p}\Lambda_1\langle -p+q\rangle)
\end{align*}

Therefore, by Theorem~\ref{Zclustertilt}, we get a triangle equivalence $F \colon \Dd^b(\cov{\Lambda_1}{\ZZ})\rightarrow \Dd^b(\cov{\Lambda_2}{\ZZ})$ which extends $f$.

\bigskip

\noindent\textit{Proof of (2).}
Now for $i=1,2$ let $\Xx_i:=\Cc^b(\proj \cov{\Lambda_i}{\ZZ})_{dg}$ be the DG category of bounded complexes of projective $\cov{\Lambda_i}{\ZZ}$-modules. The functor $$\xymatrix{F \colon \Dd^b(\cov{\Lambda_1}{\ZZ})=H^0(\Xx_1)\ar[r] & \Dd^b(\cov{\Lambda_2}{\ZZ})=H^0(\Xx_2)}$$ can be seen as $H^0(F_{dg})$ where $F_{dg}:=-\ten_{\cov{\Lambda_1}{\ZZ}} P$ and $P$ is a projective resolution of $\bigoplus_{p\in \ZZ}\SSS^{-p}T^2\langle -p\rangle$ as $\cov{\Lambda_1}{\ZZ}$-$\cov{\Lambda_2}{\ZZ}$-bimodule. 

For $i=1,2$ we set $S_i:= -\ten_{\cov{\Lambda}{\ZZ}}X_i$ where $X_i$ is a projective resolution of $D \cov{\Lambda_i}{\ZZ} [-2]$ as $\cov{\Lambda_i}{\ZZ}$-bimodule. 

In order to prove that we have an isomorphism  
  $$F_{dg}\circ S_1\simeq S_2\circ F_{dg}  \quad\textrm{ in }\rep(\Xx_1,\Xx_2),$$ it is enough to prove that we have an isomorphism  $$X_1 \ten_{\cov{\Lambda_1}{\ZZ}} P\simeq P\ten_{\cov{\Lambda_1}{\ZZ}}X_2 \quad\textrm{ in } \Dd^b((\cov{\Lambda_1}{\ZZ})^{\rm op}\ten \cov{\Lambda_2}{\ZZ}),$$ and the proof is similar to the proof of Corollary~\ref{derivedeqclustereq}.

Now, by Proposition~\ref{universalprop2}, we get a DG functor $F_S \colon \pretr(\Xx_1/S_1)\rightarrow \pretr(\Xx_2/S_2)$ such that $H^0(F_S)$ is an equivalence and such that  the following diagram commutes 
\[ \xymatrix{
{\begin{array}{c} \Dd^b(\cov{\Lambda_1}{\ZZ}) \\ =H^0(\Xx_1) \end{array}} \ar[d]^{H^0(\pi_{\Xx_1})}\ar[rr]^{F=H^0(F_{dg})} && {\begin{array}{c} \Dd^b(\cov{\Lambda_2}{\ZZ})\\ =H^0(\Xx_2)\ar[d]^{H^0(\pi_{\Xx_2})} \end{array}} \\ {\begin{array}{c} (\Dd^b(\cov{\Lambda_1}{\ZZ})/\SSS)_{\Delta} \\ =H^0(\pretr(\Xx_1/S_1)) \end{array} } \ar[rr]_{H^0(F_S)} && { \begin{array}{c} (\Dd^b(\cov{\Lambda_2}{\ZZ})/\SSS)_{\Delta} \\ =H^0(\pretr(\Xx_2/S_2)) \end{array} } } \]

For $i=1,2$ we set $T_i:=\langle 1 \rangle \colon \Xx_i\rightarrow \Xx_i$. 
It is immediate to check that we have $$S_i\circ T_i\simeq T_i\circ S_i\quad \textrm{ in }\rep(\Xx_i,\Xx_i).$$ Moreover, the functor $T_i$, as DG equivalence of $\Xx_i$ induces a  DG functor $T_i$ on $\pretr(\Xx_i/S_i)$ such that $\pi_{\Xx_i}\circ T_i\simeq T_i\circ \pi_{\Xx_i}$.

We have the following isomorphisms in $\Dd((\cov{\Lambda_1}{\ZZ})^{\rm op}\ten \cov{\Lambda_2}{\ZZ}),$

\begin{align*}
\leftsub{\rm id}{\cov{\Lambda_1}{\ZZ}}_{\langle -1\rangle} \ten_{\cov{\Lambda_1}{\ZZ}} P & \simeq \leftsub{\rm id}{\cov{\Lambda_1}{\ZZ}}_{\langle -1\rangle} \lten_{\cov{\Lambda_1}{\ZZ}}\bigoplus_{q\in\ZZ}\SSS^{-q}T^2\langle -q\rangle \\
& \simeq \bigoplus_{q\in\ZZ}\SSS^{-q+1}T^2\langle -q+1\rangle \\
& \simeq \SSS(\bigoplus_{q\in \ZZ}\SSS^{-q}T^2\langle -q\rangle) \langle 1\rangle \\
& \simeq P\langle 1\rangle \ten_{\cov{\Lambda_2}{\ZZ}}X_2
\end{align*}
and hence an isomorphism $$F_{dg}\circ T_1^{-1}\simeq S_2\circ T_2\circ F_{dg}, \quad \textrm{ in } \rep(\Xx_1,\Xx_2).$$
Therefore we have isomorphisms in $\rep(\Xx_1,\pretr(\Xx_2/S_2))$
\begin{align*}
 F_S \circ T_1^{-1} \circ \pi_{\Xx_1} & \simeq F_S \circ \pi_{\Xx_1} \circ T_1^{-1} \\
& \simeq \pi_{\Xx_2}\circ F_{dg}\circ T_1^{-1}\\
& \simeq \pi_{\Xx_2}\circ S_2\circ T_2\circ F_{dg} \\
& \simeq \pi_{\Xx_2}\circ T_2\circ F_{dg} \\
& \simeq T_2 \circ F_s \circ \pi_{\Xx_2},
\end{align*}
and we deduce an isomorphism  
$$F_S\circ T_1^{-1}\simeq T_2\circ F_S \quad \textrm{ in }\rep(\pretr(\Xx_1/S_1),\pretr(\Xx_2,S_2)).$$

Applying again Proposition~\ref{universalprop2} we get a DG functor $F_{S,T} \colon \pretr(\pretr(\Xx_1/S_1)/T_1)\rightarrow \pretr(\pretr(\Xx_2/S_2)/T_2)$ such that $H^0(F_{S,T})$ is a triangle equivalence and such that the following diagram commutes
$$\xymatrix{\pretr(\Xx_1/S_1)\ar[rr]^{F_S}\ar[d]^{\pi_{\Xx_1/S_1}} && \pretr(\Xx_2/S_2)\ar[d]^{\pi_{\Xx_2/S_2}}\\ 
\pretr(\pretr(\Xx_1/S_1)/T_1))\ar[rr]^{F_{S,T}}&& \pretr(\pretr(\Xx_2/S_2)/T_2)}.$$
Applying $H^0$ and using Corollary~\ref{doubleorbit} we obtain a commutative diagram 
$$\xymatrix{\Dd^b(\cov{\Lambda_1}{\ZZ})\ar[d]^{H^0(\pi_{\Xx_1})}\ar[rr]^{F=H^0(F_{dg})} && \Dd^b(\cov{\Lambda_2}{\ZZ})\ar[d]^{H^0(\pi_{\Xx_2})} \\ (\Dd^b(\cov{\Lambda_1}{\ZZ})/\SSS)_{\Delta}\ar[rr]_{H^0(F_S)}\ar[d]^{H^0(\pi_{\Xx_1/S_1})} && (\Dd^b(\cov{\Lambda_2}{\ZZ})/\SSS)_{\Delta}\ar[d]^{H^0(\pi_{\Xx_2/S_2})}\\ \Cc_{\Lambda_1}\simeq \left((\Dd^b(\cov{\Lambda_1}{\ZZ})/\SSS)_{\Delta}/\langle 1 \rangle\right)_{\Delta}\ar[rr]^{H^0(F_{S,T})} &&\left((\Dd^b(\cov{\Lambda_1}{\ZZ})/\SSS)_{\Delta} / \langle 1 \rangle\right)_{\Delta} \simeq \Cc_{\Lambda_2} }$$
\end{proof}

\begin{exa}
Let $H=kQ$ and $\Lambda_3=kQ^3/I^3$ be the algebras given by the following quivers (we keep the notation of Example~\ref{example}):
$$ Q=\xymatrix@-.4cm{ &2\ar[dl]_\alpha &\\ 1&&3\ar[ll]^\gamma \ar[ul]_\beta },\textrm{ and }  Q^3=\xymatrix@-.4cm{ &2\ar[dr]^b&\\ 1\ar@{..}@/^6mm/[rr]\ar[ur]^a&&3\ar[ll]^d}$$
with relations  $I^3=\langle ba\rangle$.
The graded QP associated with these algebras of global dimension~$\leq 2 $ are $(Q,0,d)$ and $(\widetilde{Q}_3,W_3,d_3)$ given by 
$$Q=\xymatrix{ &2\ar[dl]|0_(.4)\alpha &\\ 1&&3\ar[ll]|0^(.4)\gamma \ar[ul]|0_(.4)\beta } \textrm{ and } \widetilde{Q}^3= \xymatrix{ &2\ar[dr]|{0}^(.4)b&\\ 1\ar[ur]|{0}^(.6)a&&3\ar@<1mm>[ll]|(.4){0}^(.6)d\ar@<-1mm>[ll]|(.6){1}_(.4)c}$$ with $W_3=cba$. It is immediate to see that there is an isomorphism $\mu_2(Q,0)\simeq(\widetilde{Q}^3,W_3)$ which sends $\alpha^*$ on $a$, $\beta^*$ on $b$, $[\alpha\beta]$ on $c$ and $\gamma$ on $d$. If we compute $\mu_2^L(Q,0,d)$ we get the following grading on $\mu_2(Q)$:
 $$\mu_2(Q)= \xymatrix{ &2\ar[dr]|{1}^(.4){\beta^*}&\\ 1\ar[ur]|{0}^(.6){\alpha^*}&&3\ar@<1mm>[ll]|(.4){0}^(.6)\gamma\ar@<-1mm>[ll]|(.6){0}_(.4){[\alpha\beta]}}$$
Then we get a $\ZZ^2$-grading on $\widetilde{Q}^3$ given by $(\mu_2^L(d),d_3)$:
$$\widetilde{Q}^3= \xymatrix{ &2\ar[dr]|{(1,0)}^(.4)b&\\ 1\ar[ur]|{(0,0)}^(.6)a&&3\ar@<1mm>[ll]|(.4){(0,0)}^(.6)d\ar@<-1mm>[ll]|(.6){(0,1)}_(.4)c}$$ which makes the potential $W_3$ homogeneous of degree $(1,1)$ and 
which induces a grading on $\Lambda_3$ 
$$ Q^3=\xymatrix@-.3cm{ &2\ar[dr]|1^(.4)b&\\ \ar@{..}@/^6mm/[rr]1\ar[ur]|0^(.6)a&&3.\ar[ll]|0^(.4)d}$$
Now $\mu_2^R(\widetilde{Q}_3,W_3,d_3)$ yields a $\ZZ$-grading on $Q$
$$Q=\xymatrix@-.3cm{ &2\ar[dl]|1_(.4){a^*} &\\ 1&&3\ar[ll]|0^(.4)d \ar[ul]|0_(.4){b^*}. }$$
Here are the $\ZZ$-coverings of the $\ZZ$-graded algebras $H$ and $\Lambda_3$
$$\xymatrix@-.4cm{ &&&\\&&2\ar@{..}[ul]&\\& 1&&3\ar[ll]_\gamma\ar[ul]_\beta\\& &2\ar[ul]^\alpha&\\ \cov{H}{\ZZ} =&1&&3\ar[ul]_\beta\ar[ll]^\gamma \\&&2\ar[ul]^\alpha&\\& 1&&3\ar[ul]_\beta\ar[ll]_\gamma\\&&\ar@{..}[ul]&}\quad  \xymatrix@-.4cm{&&&\\ &&2\ar@{..}[ur]&\\ &1\ar[ur]^a&&3\ar[ll]_d\\ &&2\ar[ur]_b&\\\textrm{and}\quad \cov{\Lambda_3}{\ZZ}=& 1\ar@{..}@/^5mm/[uurr]\ar[ur]^a&&3\ar[ll]_d\\ &&2\ar[ur]_b&\\& 1\ar@{..}@/^5mm/[uurr]\ar[ur]^a&&3\ar[ll]_d\\& &\ar@{..}[ur]&} $$
By Theorem~\ref{maintheorem} we have a derived equivalence $\Dd^b(\cov{H}{\ZZ})\simeq \Dd^b(\cov{\Lambda_3}{\ZZ}).$
\end{exa}

\begin{rema}
\begin{enumerate}
\item
The algebras given by the quivers $$\xymatrix@-.4cm{ 1&&\\ &2\ar[ul]^\alpha&\\ 1&&3\ar[ul]_\beta\ar[ll]^\gamma} \textrm{ and } \xymatrix@-.4cm{1&&3\ar[ll]_d\\ &2\ar[ur]_b&\\ 1\ar@{..}@/^5mm/[uurr]\ar[ur]^a&&}$$ and the relation $ba=0$ are derived equivalent, one can pass from one to the other by doing the left mutation in the derived category at vertex 2. Using this repeatedly one can also directly check that we have a derived equivalence $\Dd^b(\cov{H}{\ZZ})\simeq \Dd^b(\cov{\Lambda_3}{\ZZ})$.
\item In the paper \cite{AO3}, we use Theorem~\ref{maintheorem} to deduce the shape of AR-quiver of the derived category $\Dd^b(\Lambda_3)$, and of any algebra which is cluster equivalent to the path algebra of a quiver of type $\tilde{A}_n$.
\end{enumerate}
\end{rema}

\newcommand{\etalchar}[1]{$^{#1}$}
\providecommand{\bysame}{\leavevmode\hbox to3em{\hrulefill}\thinspace}
\providecommand{\MR}{\relax\ifhmode\unskip\space\fi MR }
\providecommand{\MRhref}[2]{%
  \href{http://www.ams.org/mathscinet-getitem?mr=#1}{#2}
}
\providecommand{\href}[2]{#2}

\end{document}